\documentclass[a4paper]{amsart}
\usepackage[T1]{fontenc}
\usepackage{geometry}
\usepackage{amsfonts,amssymb,amsthm}
\usepackage{mathtools}
\usepackage{braket}
\usepackage{float}
\usepackage{dsfont}
\usepackage{multirow}

\theoremstyle{plain}
\newtheorem{theorem}{Theorem}[section]
\newtheorem{lemma}[theorem]{Lemma}
\newtheorem{proposition}[theorem]{Proposition}
\newtheorem{corollary}[theorem]{Corollary}

\theoremstyle{definition}

\newtheorem{example}[theorem]{Example}
\newtheorem{assumption}[theorem]{Assumption}
\newtheorem{remark}[theorem]{Remark}

\newfloat{stages}{H}{lol}[section]
\floatname{stages}{List}

\newcommand*{\argmin}{\operatornamewithlimits{arg\,min}}
\newcommand*{\argmax}{\operatornamewithlimits{arg\,max}}

\newcommand{\R}{\mathbb{R}}
\newcommand{\N}{\mathbb{N}}

\newcommand{\F}{\mathcal{F}}
\newcommand{\half}{\frac{1}{2}}

\newcommand{\abs}[1]{\left\lvert#1\right\rvert} 
\newcommand{\indic}[1]{\mathds{1}_{#1}} 

\newcommand{\E}{\mathsf{E}}

\DeclareMathOperator{\cost}{cost}
\DeclareMathOperator{\const}{const}
\newcommand{\Rd}{\mathbb{R}^{d}}
\newcommand{\vopt}{v^{*}}
\newcommand{\copt}{c^{*}}

\newcommand{\vmc}{v}
\newcommand{\cmc}{c}

\numberwithin{equation}{section}

\begin{document}
\author{C.~Bayer}
\address{Christian Bayer \\ Weierstrass Institute \\ Berlin, Germany}
\email{christian.bayer@wias-berlin.de}

\author{D.~Belomestny}
\address{Denis Belomestny \\ Duisburg-Essen University, Germany and National University Higher School of Economics, Russia}
\email{denis.belomestny@uni-due.de}

\author{P.~Hager}
\address{Paul Hager \\
Institut f\"ur Mathematik, Humboldt Universit\"at zu Berlin \\ Germany}
\email{paul.hager@hu-berlin.de}

\author{P.~Pigato}
\address{Paolo Pigato \\ Department of Economics and Finance \\ University of Rome Tor Vergata \\ Italy}
\email{paolo.pigato@uniroma2.it}

\author{J.~Schoenmakers}
\address{John Schoenmakers \\ Weierstrass Institute \\ Berlin, Germany}
\email{john.schoenmakers@wias-berlin.de}

\author{V.~Spokoiny}
\address{Vladimir Spokoiny \\ Weierstrass Institute \\ Berlin, Germany}
\email{vladimir.spokoiny@wias-berlin.de}

\begin{abstract}
  Least squares Monte Carlo methods are a popular numerical approximation method for solving stochastic control problems. Based on dynamic programming, their key feature is the approximation of the conditional expectation of future rewards by linear least squares regression. Hence, the choice of basis functions is crucial for the accuracy of the method. Earlier work by some of us [Belomestny, Schoenmakers, Spokoiny, Zharkynbay. Commun.~Math.~Sci., 18(1):109–121, 2020] proposes to \emph{reinforce} the basis functions in the case of optimal stopping problems by already computed value functions for later times, thereby considerably improving the accuracy with limited additional computational cost. We extend the reinforced regression method to a general class of stochastic control problems, while considerably improving the method's efficiency, as demonstrated by substantial numerical examples as well as theoretical analysis.
\end{abstract}

\thanks{C.B., P.H, P.P. J.S. and V.S. were supported by the MATH+ project AA4-2 Optimal control in energy markets using rough analysis and deep networks.
The authors are also thankful to the associate editor and anonymous referees for there feedback and suggestions.}



\title{Reinforced optimal control}

\maketitle

\section{Introduction}
\label{sec:introduction}

Stochastic control problems form an important class of stochastic optimization problems that find applications in a wide variety of fields, see \cite{pham2009continuous} for an overview. The general problem can be formulated as follows: How should a decision-maker control a system with a stochastic component to maximize the expected reward?
In the theory of stochastic control, one distinguishes between problems with continuous and discrete sets of possible control values. While the first class of control problems contains, for example, energy storage problems, the second one includes stopping and multiple stopping problems. Furthermore one differentiates between discrete-time and continuous-time optimal control problems.
(Neither of these distinctions is fundamental: for instance, many numerical methods will replace optimal control problems with a continuous set of control values in continuous time by a surrogate problem with discrete control values in discrete time. Moreover, many discrete optimal control problems may well be analyzed as continuous ones, if the number of possible control values or time-steps is finite but very high).
\par
The range of applications of stochastic control problems is very wide. Originally, optimal stochastic continuous control problems were inspired by engineering problems in the continuous control of a dynamic system in the presence of random noise, see \cite{aastrom2012introduction} and references therein. In the last decades, problems in mathematical finance (portfolio optimization, options with variable exercise possibilities) and economics inspired many new developments, see \cite{bauerle2011markov} for some recent developments.  Let us also mention a closely connected area of reinforcement learning with plethora of applications in robotics, data science, and engineering, see  \cite{sutton2018reinforcement}.
\par
As a canonical general approach for solving  a discrete-time optimal control problem one may  consider all possible future
evolutions of the process at each time that a control choice is to be made, see \cite{aastrom2012introduction}.
This method is well developed and may be effective in some special cases but for more general problems such as optimal control of diffusion in high dimensions,
this approach is impractical.  
In \cite{belomestny2009regession} a generic  Monte Carlo approach
combined with linear regression was proposed and studied, see also \cite{BSbook} for an overview.
However, as an important
disadvantage,  there may be not enough flexibility when modeling highly non-linear behavior of optimal value functions.
For instance, a regression based on higher-degree polynomials or local polynomials (splines) may contain too many
parameters and, therefore, may over-fit the Monte Carlo sample or even prohibit
parameter estimation because the number of parameters is too large.  As an alternative to the polynomial bases,  nonlinear approximation structures (e.g., artificial neural networks) can be used instead (see, e.g. \cite{han2016deep},\cite{becker2019deep} and \cite{becker2019solving}).

\par
In \cite{belomestny2020optimal}  a {Monte Carlo based {\em reinforced regression}  approach} is developed for building sparse
regression models at each backward step of the dynamic programming algorithm in the case of optimal stopping problems.
In a nutshell, the idea is to start with a generic set of basis functions, which is systematically enlarged with highly problem-dependent additional functions.
The additional basis functions are constructed  for the
optimal stopping problem at hand without using a fixed predefined finite
dictionary. The new basis functions are learned during the
backward induction via incorporating information from the preceding backward
induction step. More specifically, the (computed, hence approximate) value function at time $t_{i+1}$ is used as an additional basis function at time $t_i$.
Thereby, basis functions highly specific to the problem at hand are constructed in a completely automatic way. Indeed, the continuation function at time $t_i$ can often be observed to be very close to the value function at time $t_{i+1}$, especially when the time-step $t_{i+1} - t_i$ is small -- alluding to continuity in time of the solution to some continuous time version of the optimal stopping problem.
\cite{belomestny2020optimal} report that the reinforced basis leads to increased precision over the starting set of basis functions, comparable to the standard regression algorithm based on a substantially increased set of basis functions. This improvement is obtained with a limited increase of the computational cost.
\par
In this work, we carry over the approach of \cite{belomestny2020optimal} to a general class of
discrete-time optimal control problems including multiple stopping problems
(thus allowing pricing of swing options) and a gas storage problem. This
generalization turns out to be rather challenging as the complexity of using
the previously constructed value function in regression basis at each step of
the backward procedure becomes prohibitive when applying the original approach
of \cite{belomestny2020optimal}. We overcome this computational bottleneck by introducing a novel version of the original reinforced regression algorithm where one uses a
hierarchy of fixed time-depth approximation of the optimal value function instead of a full-depth approximation employed in \cite{belomestny2020optimal}. As a result, we regain
efficiency and are able to improve upon the standard linear regression algorithm in terms of achievable precision for a given computational budget.
\par
More precisely, we construct a hierarchy $v^{(i)}$, $i = 0, \ldots, I$, of (approximate) value functions with depth $I > 0$. Here, $v^{(0)}$ denotes the value functions obtained from the classical Monte Carlo regression algorithm. The higher levels $v^{(i)}$ are computed by regression based on a set of basis functions reinforced by the value function $v^{(i-1)}$ one level lower. This way, the added computational cost incurred from reinforcing the basis can be further decreased with minimal sacrifices of accuracy already for small values of $I$.
In fact, we propose two versions of the algorithm. In the first version, the levels of the hierarchy of value functions are trained consecutively, allowing for an adaptive choice of the depth $I$ of the hierarchy. In the second version, all the levels are trained concurrently, thereby improving the accuracy at each individual level. As a consequence, $I$ needs to be fixed in advance and cannot be chosen adaptively in the second variant.

\subsection*{Outline of the paper}
\label{sec:outline-paper}

In Section~\ref{sec:setting} we describe a rather general setting for discrete stochastic control problems which we are going to use in this paper. The setting is based on~\cite{gyurko2011monte}. We recall the reinforced regression algorithm for optimal stopping problems by~\cite{belomestny2020optimal} in detail in Section~\ref{sec:reinf-regr}. There we also motivate the hierarchical construction of the new reinforced regression algorithm as restricted to the optimal stopping problem. The full algorithm -- including both variants -- is introduced in Section~\ref{sec:altern-algor}. A detailed analysis of computational costs is provided in Section~\ref{sec:computational-cost}. The next Section~\ref{sec:convergence-analysis} provides a detailed convergence analysis for the standard and reinforced regression algorithms in the current setting. Extensive numerical examples including optimal stopping problems, multiple stopping problems and a gas storage optimization problem are provided in Section~\ref{sec:numerical-examples}.

\section{Setting}
\label{sec:setting}

First, we present a proper setting for the construction and analysis of reinforced regression algorithms. The setting will be largely based on~\cite{gyurko2011monte}. We will consider stochastic control problems in discrete time with finite action sets. We note that extensions to continuous action sets are certainly possible, but are left to future research.

We consider a filtration $\F_j$, $j=0, \ldots, J$, which is extended by
$\F_{-1} \coloneqq \set{\emptyset, \Omega}$, $\F_{J+1} \coloneqq \F_J$ for
convenience. Let $X$ be a Markov process with values in \(\mathcal{X}\) adapted to
$(\F_j)_{j=0, \ldots, J}$. Note that we assume that the dynamics of the underlying process $X$ does not depend on the control.

At time $0 \le j \le J$ we are given a \emph{control} $Y_j$, which is
$\F_{j-1}$-measurable, and an $\F_j$-measurable \emph{cash-flow} $Z_j = H_j(a,
Y_j, X_j)$ for some deterministic, measurable function $H_j$, where $a$ is an
\emph{action} that we may choose at time $j$ in some finite action space $\mathcal{K}$. Note that cash-flows may be
positive or negative. We assume that the control $Y_j$ takes values in a
finite set $\mathcal{L}$.

\begin{remark}
  \label{rem:finite-action}
  The assumption that actions $a$ and controls $y$ take values in finite sets $\mathcal{K}$ and $\mathcal{L}$, respectively, is a weaker assumption than it may seem at first sight. Many important control problems naturally fall into this class, see examples below. Even more importantly, it is a well-known fact that many optimal control problems with genuinely continuous action and control spaces have solutions of \emph{bang-bang} type, i.e., all optimal controls consist of actions taken from a finite set, usually at the boundaries of the (continuous) action sets. Hence, such control problems can effectively be reduced to control problems with finite actions sets. Extensions of the reinforced regression algorithm to infinite action spaces will be studied in future work.
\end{remark}
For a given value of the control $y \in \mathcal{L}$ and a given
value $x$ of the underlying process $X_j$, we are given a set of admissible actions
\begin{equation}
  \label{eq:1}
  K_j(y,x) \subset \mathcal{K}, \quad j=0, \ldots, J,
\end{equation}
i.e., $a$ is admissible iff $a \in K_j(x,y)$.
Finally, if we apply $a
\in K_j(Y_j, X_j)$, then the control is updated by
\begin{equation}
  \label{eq:2}
  Y_{j+1} \coloneqq \varphi_{j+1}(a, Y_j), \quad \varphi_{j+1}: \mathcal{K}
  \times \mathcal{L} \to \mathcal{L}.
\end{equation}
Suppose that the control and the
underlying state process take values $Y_j$ and $X_j$ at time  $0 \le j \le J$, respectively.
For $\mathbf{a} \coloneqq (a_j, \ldots, a_J) \in \mathcal{K}^{J-j+1}$ and $j \le \ell \le J-1$, we define
\begin{equation}
  \label{eq:3}
  Y_{\ell+1}(\mathbf{a};j,Y_j) \coloneqq \varphi_{\ell+1}(a_\ell, Y_\ell(\mathbf{a};j,Y_j)),
  \quad Y_j(\mathbf{a}; j, Y_j) \coloneqq Y_j,
\end{equation}
noting that $Y_{\ell}(\mathbf{a}; j,Y_j)$ only depends on $a_j, \ldots,
a_{\ell-1}$. Additionally, we define $\mathbb{F}_{j,J}(\mathcal{K})$ to be the set
of $(\F_\ell)_{\ell=j}^J$-adapted processes taking values in $\mathcal{K}$
indexed by $j, \ldots, J$. Clearly, if $\mathbf{A} \coloneqq (A_\ell)_{\ell=j}^J \in
\mathbb{F}_{j,J}(\mathcal{K})$ and $Y_j \in \F_{j-1}$, then the process
$Y_\cdot(\mathbf{A}; j, Y_j)$ is previsible. The set of
\emph{admissible strategies or admissible policies}
$\mathcal{A}_j$ is defined as follows:
\begin{equation}
  \label{eq:4}
  \mathcal{A}_j(Y_j,X_{\ge j}) \coloneqq \Bigl\{ \mathbf{A} = (A_\ell)_{\ell=j}^J \in
  \mathbb{F}_{j,J}(\mathcal{K}) \, \Big| \, A_\ell \in
  K_\ell(Y_\ell(\mathbf{A};j,Y_j),X_\ell), \quad \ell = j, \ldots, J \Bigr\}.
\end{equation}
Now the central issue is the optimal control problem
\begin{equation}
  \label{eq:6}
  V_j \coloneqq \sup_{\mathbf{A} = (A_\ell)_{\ell=j}^J \in \mathcal{A}_j(Y_j,X_{\ge j})} \E_j\left[ \sum_{\ell=j}^J
    H_\ell(A_\ell, Y_\ell(\mathbf{A}; j, Y_j), X_\ell) \right],
\end{equation}
at a generic time $0\le j\le J,$ where $\E_j$ denotes the conditional expectation w.r.t.~$\F_j$.

Taking advantage of the Markov property, we introduce the notation
$\mathcal{A}_j(y,x):= \mathcal{A}_j(y,X^x_{\ge j}),$
where $X^{j,x}$ denotes the Markov process $X$ conditioned on $X_j = x$, and is defined for $j \le \ell \le J$.
We may then define the \emph{value function} as
\begin{equation}
  \label{eq:7}
  \vopt_j(y,x) \coloneqq \sup_{\mathbf{A} = (A_\ell)_{\ell=j}^J\in \mathcal{A}_j(y, x)} \E\left[ \sum_{\ell=j}^J
    H_\ell\left(A_\ell, Y_\ell(\mathbf{A}; j, y), X_\ell^{j,x}\right)  \right],
\end{equation}
which satisfies the \emph{dynamic programming principle}:
\begin{equation}
  \label{eq:8}
  \vopt_j(y,x) = \sup_{a \in K_j(y,x)} \left( H_j(a, y, x) + \E\left[
      \vopt_{j+1}(\varphi_{j+1}(a, y), X_{j+1}^{j,x}) \right] \right).
\end{equation}
for $j=0, \ldots, J$ (with $\vopt_{J+1}(y,x) \coloneqq 0$).

Let us now give a few examples for
classical stopping and control problems which fall into the above setup.

\begin{example}
  \label{ex:single-stopping}
  For a single optimal stopping problem with payoff $g_j\ge0$ at time $j$, the set
  of possible control values is $\mathcal{L} = \set{0,1}$,  where a control state $y$
  denotes the number of remaining exercise opportunities. The action $a$ takes
  the value $1$ if we stop at the current time and $0$ otherwise.
  Hence, we have
  \begin{equation*}
    K_j(y,x) = K(y) \coloneqq
    \begin{cases}
      \set{0,1}, & y = 1,\\
      \set{0}, & y = 0,
    \end{cases}
  \end{equation*}
  implying that $\mathcal{K} = \set{0,1}$. The cash-flow is defined by
  \begin{equation*}
    H_j(a,y,x) \coloneqq a \, g_j(x),
  \end{equation*}
  independent of the value of the control $y$. Finally, the update function
  of the control is defined by $\varphi_{j+1}(a,y) \coloneqq \max(y-a, 0).$ Note that the value function $v_j^\ast(0,\cdot) \equiv 0$, and, hence, the optimal stopping literature usually only considers $(j,x) \mapsto v_j^\ast(1,x)$.
\end{example}

\begin{example}
  \label{ex:multiple-stopping}
  Let us now suppose that we have a multiple stopping problem with $L \in \N$
  exercise rights. Again, the control state $y$ signifies the remaining exercise
  opportunities, leading to $\mathcal{L} = \set{0,1, \ldots, L}$. The
  admissible action  set is now defined as
  \begin{equation*}
    K_j(y,x) = K(y) \coloneqq
    \begin{cases}
      \set{0,1}, & y \ge 1,\\
      \set{0}, & y = 0.
    \end{cases}
  \end{equation*}
  Again, $\mathcal{K} = \set{0,1}$. The cash-flow $H_j$ and the
  update function $\varphi_{j+1}$ are defined as in
  Example~\ref{ex:single-stopping}.
\end{example}

\begin{example}
  \label{ex:gas-storage}
  Consider a simple gas storage problem: given $N \in \N$ and $\Delta = 1/N$,
  we assume that the volume of gas in a storage can only be increased and decreased
  by a fraction $\Delta$ over a given time increment. Let the control $y$
  denote the status (fill level) of the storage at time $j$. Hence, we define $\mathcal{L} \coloneqq \set{0, \Delta, 2\Delta, \ldots, 1}.$
  At time $j$, we may either sell $\Delta$ (volume of gas; $a = -1$), buy
  $\Delta$ ($a = +1$) -- at the current market price $X_j$ -- or do nothing
  ($a = 0$). Hence, the admissible policy set is
  \begin{equation*}
    K_j(y) \coloneqq
    \begin{cases}
      \set{0,1}, & y = 0, \\
      \set{-1,0,1}, & \Delta \le y \le 1-\Delta, \\
      \set{-1,0}, & y = 1,
    \end{cases}
  \end{equation*}
  with $\mathcal{K} = \set{-1,0,1}$, while the cash-flow is given by
  \begin{equation*}
    H_j(a, y, x) \coloneqq -a \Delta x.
  \end{equation*}
  The update function in given by $ \varphi_{j+1}(a,y) \coloneqq ((y+a\Delta)\wedge 1) \vee 0.$
\end{example}

\begin{remark}
  \label{rem:optimal-control}
  While we do not allow the actions to have an effect on the dynamics of the state process $X$, a large class of more general control problems could be incorporated by a simple modification of our setting. If we allow updates of the control variable $y$ to depend on the state $x$ as well as on the previous control and the action, i.e., $Y_{j+1} = \varphi_j(a, Y_j, X_j)$, then our theoretical analysis remains intact. However, it now becomes possible to control the dynamics of the state process $X$, provided that the law of the controlled process remains absolutely continuous w.r.t.~the law of the original process $(X_j)_{j=0, \ldots, J}$. We refer to the discussion of the optimal liquidation example in \cite[Section 2]{gyurko2011monte} for more details. Note, however, allowing $Y$ to depend on $X$ in such a way might require us to use regression in $(x,y)$ rather than just $x$ for most practical problems.
\end{remark}

\section{Reinforced regression for optimal stopping}
\label{sec:reinf-regr}

In this section, we recall the standard regression algorithm as well as the  reinforced regression algorithm introduced in \cite{belomestny2020optimal} for optimal stopping problems. We will point out the drawbacks of the latter algorithm for more general control problems, and propose and motivate several  modifications.
However, for the purpose of a clear illustration, we will restrict ourselves in this section to the optimal stopping case.

Let us recall the optimal stopping setup from Example \ref{ex:single-stopping} and denote by $\vopt_j(x)$ the value function at $j \in \{0, ..., J\}$ evaluated in $x\in\Rd$ and $y = 1$.
Further recall that the dynamic programming principle is given by
\begin{equation*}
\vopt_j(x) = \max(g_j(x), \copt_j(x)),\;\; 0\le j \le J-1,\quad \vopt_J(x)=g_J(x), \quad x\in\Rd,
\end{equation*}
 where the continuation function is given by $\copt_j(x) = \E_j[\vopt_{j+1}(X^{j,x}_{j+1})]$.
Fix a set of basis function $\{\psi_1, ..., \psi_K\}$ with $\psi_k: \Rd \to \R$, $k=1, \ldots, K$, and sample trajectories $(X^{(m)}_j)_{0\le j \le J, 1\le m \le M}$ from the underlying Markov chain, i.e., $(X^{(m)}_j)_{0\le j \le J}$ are i.i.d.~samples from the distribution of $(X_j)_{0 \le j \le J}$, $m=1, \ldots, M$. Then the regression method due to Tsitsiklis-van Roy \cite{tsitsiklis2001regression}, which we will refer to as the \emph{standard regression method}, inductively constructs an approximation $\vmc = (\vmc_j)_{j=0,...,J}$ to the value function $\vopt$ as follows:
For $j=J$ initialize $\vmc_J = g_J$.
For $j \in \{J-1, ..., 0\}$ set
\begin{align}\label{eq:optimal_stopping_bellman}
\vmc_j(x) := \max(g_j(x), \cmc_j(x)), \quad \cmc_j(x) = \sum_{k=1}^{K} \gamma_{j,k} \psi_k(x),
\end{align}
where the regression coefficients are given by the solution to the least squares problem
\begin{align}\label{eq:optimal_stopping_least_squares}
\gamma_{j,1}, ..., \gamma_{j,K} := \underset{\gamma_{1}, ..., \gamma_{K}}{\argmin}\sum_{m=1}^{M}\bigg\vert \vmc_{j+1}(X^{(m)}_{j+1}) - \sum_{k=1}^{K} \gamma_k \psi_k(X^{(m)}_j) \bigg\vert^{2}.
\end{align}
The procedure is illustrated in Figure \ref{fig:standard_regression}.
Note that the costs of this algorithm are of the order $M \cdot J \cdot K^{2}$ (see, e.g., \cite{belomestny2020optimal} or Section \ref{sec:computational-cost}).

\begin{figure}[h]\centering
\includegraphics[scale=0.4]{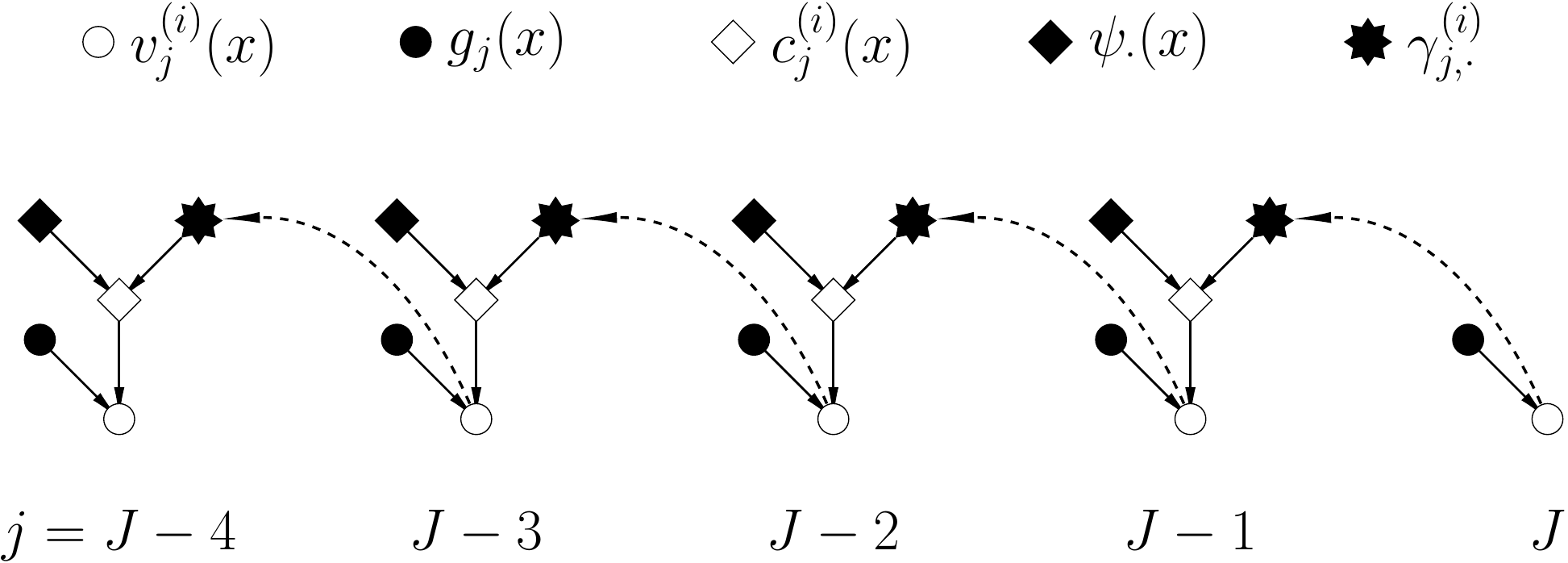}
\caption{Illustration of standard regression approach due to Tsitsiklis-van Roy \cite{tsitsiklis2001regression}.
The solid arrows indicate the dependencies in the (feed forward) evaluation of $\cmc_j$ and $\vmc_j$ in \eqref{eq:optimal_stopping_bellman}.
The dashed arrows start from the regression data $\vmc_{j+1}$ and symbolize the regression procedure \eqref{eq:optimal_stopping_least_squares}.}\label{fig:standard_regression}
\end{figure}

One problem of the standard regression algorithm is that its performance strongly depends on the choice of basis functions. Indeed, while standard classes such as polynomials or splines usually form the backbone of the construction of basis functions, practitioners usually add customized basis functions, for instance the payoff function $g_j$ and some functionals applied to it.

As a more systematic approach, the authors of \cite{belomestny2020optimal} proposed a \emph{reinforced regression algorithm}.
In this procedure the regression basis at each step of the backward induction is reinforced with the approximate value function from the previous step of the induction.
The approximate continuation function at $j\in\{0, ..., J-1\}$ is then given by
\begin{equation*}
\cmc_j(x) \coloneqq \sum_{k=1}^{K}\gamma_{j,k} \psi_k(x) + \gamma_{j, K+1} \vmc_{j+1}(x),
\end{equation*}
where the regression coefficient are the solution to the least squares problem
\begin{equation*}
\gamma_{j,1}, ..., \gamma_{j,K+1} \coloneqq \underset{\gamma_1, ..., \gamma_{K+1}}{\argmin}\sum_{m=1}^{M}\bigg\vert \vmc_{j+1}(X^{(m)}_{j+1}) - \sum_{k=1}^{K} \gamma_k \psi_k(X^{(m)}_j) - \gamma_{K+1}\vmc_{j+1}(X^{(m)}_j) \bigg\vert^{2}.
\end{equation*}
Note that this procedure induces a recursion whenever an approximate value function is evaluated: in order to evaluate $\vmc_j(x)$ we need to evaluate $\cmc_j(x)$, which in turn requires an evaluation of $\vmc_{j+1}(x)$ and so forth, until $\vmc_{J}(x) = g_J(x)$ terminates the recursion.
Figure \ref{fig:reinforced_regression} illustrates this procedure.
The costs of the reinforced regression method are proportional to $M \cdot J \cdot K^{2} + M \cdot J^{2} \cdot K$ (see \cite{belomestny2020optimal}).

\begin{figure}[h]\centering
\includegraphics[scale=0.4]{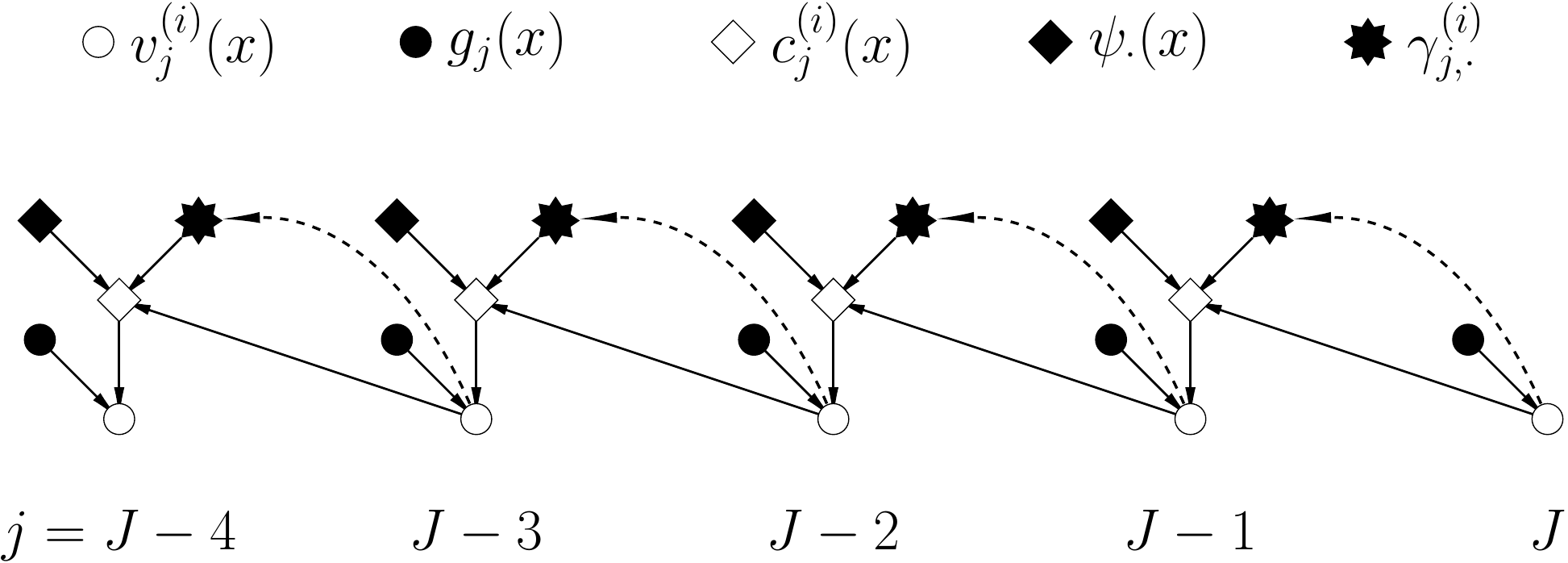}
\caption{Illustration of the reinforced regression approach.
Evaluation of $\vmc_j$ in the reinforced regression algorithm leads to a recursion with $J-j$ steps.}\label{fig:reinforced_regression}
\end{figure}

Despite the increased computational cost compared to the standard regression algorithm with the same set of basis functions $\psi_1, \ldots, \psi_K$, the reinforced regression algorithm can improve the overall computational cost for a fixed error tolerance drastically. As a rule of thumb, \cite{belomestny2020optimal} report that the reinforced regression algorithm with a standard basis consisting of polynomials of a given degree leads to similar accuracy as the standard regression algorithm based on polynomials of one degree higher. In particular, the reinforced regression algorithm already outperforms the standard regression algorithm for small dimensions $d>1$, as long as the number $J$ of time-steps is not too large.

A direct generalization of the reinforced regression algorithm to more general control problems is certainly possible. The main difference to the optimal stopping problem is that for fixed time $j$ we have to choose from many potential candidates to reinforce with, namely any $\vmc_{j+1}(y, \cdot),$ $y \in \mathcal{L}$ is a candidate. Additionally, the dynamic programming principle~\eqref{eq:8} now entails a possibly non-trivial optimization problem in terms of the policy $a$. Especially the second point makes the recursion at the heart of the reinforced regression algorithm untenable for general control problems.

One solutions immediately comes to mind: If performing the recursion all the way to terminal time $J$ is too costly, why not truncate at a certain recursion depth? This idea is, in principle, sound, and is the basis of the adaptations suggested below. However, some care is needed in the implementation of this idea. Indeed, if ``truncation'' simply were to mean ``replace the reinforcing basis functions by $0$ after a certain truncation step'', this would introduce a structural error in the procedure, as regression coefficients formerly computed in the presence of these basis functions would suddenly be incorrect. Instead, we propose to compute a \emph{hierarchy} of reinforced regression solutions, corresponding to different ``cut-off depths'' of the recursion. This way, we can make sure that the coefficients are always consistent, that is, an error as mentioned above can be avoided. We introduce two versions, which both adhere to the same general idea, but differ in an important implementation detail.

The \emph{hierarchical reinforced regression algorithm A} iteratively  constructs approximations \break $(v^{(i)})_{i=0, 1, ...}$ to the true value function as follows:
For $i=0$ we construct $(v^{(0)}_j)_{0\le j \le J}$ using the standard regression method described above.
Then for any $i \ge 1$, given that $v^{(l)}$ is already constructed for $0 \le l \le i-1$, define $v^{(i)}$ with the usual backwards induction, where the regression basis at step $j \in \{J-1, ..., 0\}$ is reinforced with $v^{(i-1)}_j$. The approximate continuation function of the $i^{th}$ iteration is given by
\begin{equation}\label{eq:optimal_stopping_adapted_continuation}
c^{(i)}_j(x) \coloneqq \sum_{k=1}^{K}\gamma^{(i)}_{j,k} \psi_k(x) + \gamma_{j,K+1}^{(i)} v^{(i-1)}_{j+1}(x),
\end{equation}
where the regression coefficients are the solutions to the least squares problem
\begin{equation*}
\gamma^{(i)}_{j,1}, ..., \gamma^{(i)}_{j,K+1} \coloneqq \underset{\gamma_1, ..., \gamma_{K+1}}{\argmin}\sum_{m=1}^{M}\bigg\vert v^{(i)}_{j+1}(X^{(m)}_{j+1}) - \sum_{k=1}^{K} \gamma_k \psi_k(X^{(m)}_j) - \gamma_{K+1}v^{(i-1)}_{j+1}(X^{(m)}_j) \bigg\vert^{2}.
\end{equation*}
The procedure may be stopped after a fixed number of iterations, or using an adaptive criterion.
An illustration of the method can be found in Figure \ref{fig:adaptive_reinforced_regression}.
Note that the recursion that is started when evaluating $v^{(i)}_j(x)$ always terminates after at most $i$ steps in the evaluation of $v^{(0)}_{j+i}(x)$ for $i \le J-j$ or in $v^{(i-J-j)}_{J}(x)=g_J(x)$ for $J-j\le i$.

\begin{figure}[h]\centering
\includegraphics[scale=0.4]{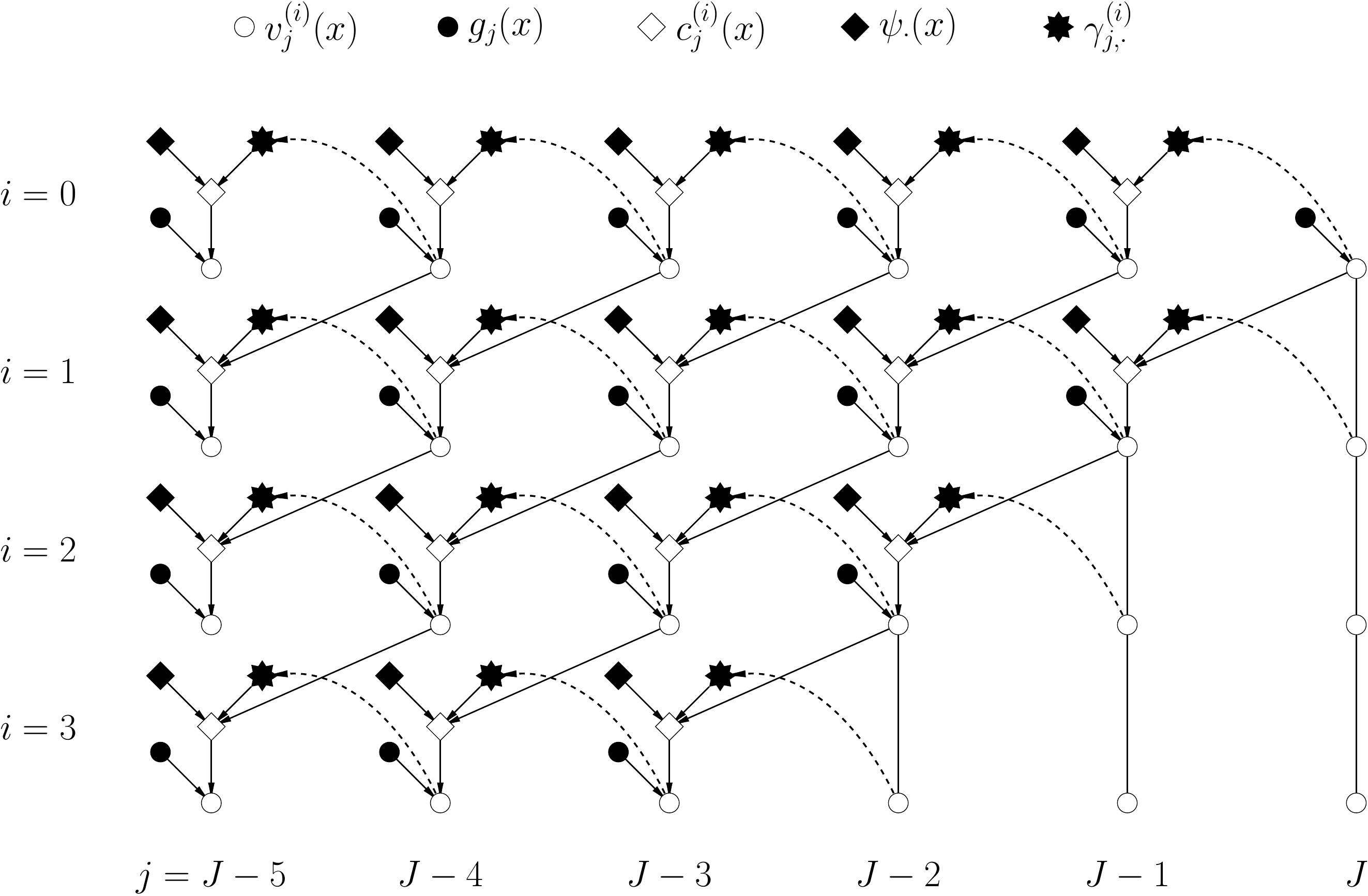}
\caption{Illustration of the hierarchical reinforced regression algorithm A, for three iterations.
In the lower right part of the diagram, the vertical lines indicate the equality $v^{(i)}_j \equiv v^{(l)}_j$ for $J-j \le i$.
}\label{fig:adaptive_reinforced_regression}
\end{figure}

For a fixed number of iterations $i \in \{0, ..., I\}$ we can modify the structure of the previous method so that the primary iteration is the backwards induction over $j \in \{ J, J-1, ..., 0\}$ and the secondary iteration is over $i \in \{ 0, ..., I\}$.
In this case we can further modify the algorithm by using $v^{(I)}_{j+1}$ as the regression target for the continuation functions $c^{(i)}_{j}$ for all $i \in \{0, ..., I\}$.
We name the resulting algorithm the \emph{hierarchical reinforced regression algorithm B}.
The approximate continuation function at step $j$ and iteration $i$ is then still given by \eqref{eq:optimal_stopping_adapted_continuation} and the least squares problem is given by
\begin{equation*}
\gamma^{(i)}_{j,1}, ..., \gamma^{(i)}_{j,K+1} \coloneqq \underset{\gamma_1, ..., \gamma_{K+1}}{\argmin}\sum_{m=1}^{M}\bigg\vert v^{(I)}_{j+1}(X^{(m)}_{j+1}) - \sum_{k=1}^{K} \gamma_k \psi_k(X^{(m)}_j) - \gamma_{K+1}v^{(i-1)}_{j+1}(X^{(m)}_j) \bigg\vert^{2}.
\end{equation*}
Also in this algorithm, the recursion that is started when evaluating $v^{(I)}_j$ stops after at most $I$ steps.
The costs of the algorithm are discussed in Section \ref{sec:computational-cost}.

\section{Iterated reinforced regression for optimal control}
\label{sec:altern-algor}

Following the ideas and motivations of Section~\ref{sec:reinf-regr} we now present hierarchical reinforced regression algorithms for optimal control based on the Bellman equation (\ref{eq:8}). The algorithms are based on  $M$ sample trajectories $(X^{(m)}_j)_{j=0, ..., J, m = 1, ..., M}$ from the underlying Markov chain $X$. and some initial set
 $\{\psi_1, ..., \psi_K\}$  of basis functions $\psi_i: \Rd \to \R$. For each $y\in\mathcal{L}$ we will define a subset $\mathcal{L}^{y} \subset \mathcal{L}$ of cardinality $R^{y} \coloneqq |\mathcal{L}^y|$ and reinforce the basis $\{\psi_1, \ldots, \psi_K\}$ by $\{v_{j+1}(z, \cdot) | z \in \mathcal{L}^y\}$.
The respective algorithms iteratively construct sequences of approximations to the value function
\begin{equation*}
v^{(i)} = (v^{(i)}_j)_{j=0,...,J} \quad \text{with} \quad  v^{(i)}_j : \mathcal{L} \times \Rd \to \R,
\end{equation*}
for $i=\{ 0, 1, ... \}$ until the iteration is terminated.

\subsection{Hierarchical reinforced regression algorithm A}\label{sec:adaptive_reinforced_regression}

For $i=0$ construct $v^{(0)}$ using the standard regression method inductively as follows:
At the terminal time $J$ initialize $v^{(0)}_J \coloneqq v_J$ where
\begin{equation}\label{eq:vJ_initialization}
v_J(y, x) = \max_{a \in K_J(y,x)} H_J(a, y, x), \quad\text{for all}\quad y\in\mathcal{L},\; x \in \Rd.
\end{equation}
For a $j\in\{0, ..., J-1\}$, assume that $v^{(0)}_l$ is already constructed for all $l \in \{j+1, ..., J\}$.
Then for each $y\in\mathcal{L}$ define the regression coefficients by solving the following least squares problem
\begin{equation}\label{eq:simple_least_squares}
\gamma^{(0), y}_{j,1}, ..., \gamma^{(0), y}_{j,K} \;\coloneqq\; \underset{\gamma_1, ..., \gamma_K}{\argmin}{\sum_{m=1}^{M}\left\vert v^{(0)}_{j+1}(y, X^{(m)}_{j+1}) - \sum_{k=1}^{K}\gamma_k\psi_k(X^{(m)}_j)\right\vert^{2}}.
\end{equation}
Next define the continuation function by
\begin{equation}\label{eq:simple_continuation_function}
c^{(0)}_j(y, x) \coloneqq \sum_{k=1}^{K} \gamma_{j,k}^{(0),y} \psi_k(x), \quad\text{for all}\quad y\in\mathcal{L},\; x \in \Rd
\end{equation}
and the approximate value function $v^{(0)}_j$ through the dynamic programming principle
\begin{equation}\label{eq:simple_value_function}
v^{(0)}_j(y,x) \coloneqq \max_{a \in K_j(y,x)} \Big(H_j(a, y, x) + c^{(0)}_j(\varphi_j(a, y), x)\Big) \quad\text{for all}\quad y\in\mathcal{L},\; x \in \Rd.
\end{equation}

Given the approximation $v^{(i)}$ for some $i\ge 0$ we construct a new approximation $v^{(i+1)}$ using reinforced regression inductively as follows: Initialize at the terminal time $v^{(i+1)}_J \coloneqq v_J$.
For $j \in \{0, ..., J-1\}$ assume that $v^{(i+1)}_l$ is already constructed for $l \in \{j+1, ..., J\}$.
Then for each $y \in \mathcal{L}$ define the regression coefficients by solving the following least squares problem
\begin{equation}\label{eq:adaptive_least_squres}
\begin{split}
\gamma^{(i+1), y}_{j,1}, ..., \gamma^{(i+1), y}_{j,K + R^{y}}\coloneqq \underset{\gamma_1, ..., \gamma_{K+R^{y}}}{\argmin}\sum_{m=1}^{M}\bigg\vert\; v^{(i+1)}_{j+1}(y, X^{(m)}_{j+1}) - \sum_{k=1}^{K}\gamma_k\psi_k(X^{(m)}_j) \\
- \sum_{k=1}^{R^{y}} \gamma_{K+k} v^{(i)}_{j+1}(y_k, X^{(m)}_j)\;\bigg\vert^{2},
\end{split}
\end{equation}
where $\{y_k\}_{k=1, ..., R^{y}} = \mathcal{L}^y$, and define the continuation function $c^{(i+1)}_j$ by
\begin{equation}\label{eq:reinfroced_continuation_function}
c^{(i+1)}_j(y, x) \coloneqq \sum_{k=1}^{K} \gamma^{(i+1),y}_{j,k} \psi_k(x) + \sum_{k=1}^{R^{y}} \gamma^{(i+1),y}_{j,K+k} v^{(i)}_{j+1}(y_k, x),
\end{equation}
for all $y\in\mathcal{L}$ and $x\in\Rd$.
Finally define the approximation $v_j^{(i+1)}$ through the dynamic programming principle by
\begin{equation}\label{eq:approximate_bellman_principle}
v^{(i+1)}_j(y,x) \coloneqq \max_{a \in K_j(y,x)} \Big(H_j(a, y, x) + c^{(i+1)}_j(\varphi_j(a, y), x)\Big),
\end{equation}
for all $y \in \mathcal{L}$ and $x \in \Rd$.

The iteration over $i\in\{0,1, ...\}$ can be terminated after $I\in\mathbb{N}$ steps, yielding $v^{(I)}$ as an approximation to the true value function.
Alternatively one can introduce an adaptive termination criterion, for example by comparing the relative change in the error of the least squares problem (\ref{eq:adaptive_least_squres}), terminating after the change falls under a given threshold.
\begin{remark}
Recall that in the initialization we have $v^{(i)}_J = v^{(0)}_J$ for all $i\in\{1, ..., I\}$.
It then follows inductively that
\begin{equation}\label{eq:identity_right_lower_triangle}
v^{(i)}_j \equiv v^{(l)}_j, \quad \text{ for all }\quad  J - j\le i \le I, \quad l \ge i.
\end{equation}
This identity can be used to reduce the costs of the algorithm, since the regression problem only needs to be solved for all $(j,i)$ with $0\le j\le J-1$ and $0 \le i \le (J-j)\wedge I$.
\end{remark}
\begin{remark}\label{rem:choice_of_Ly} More general or other forms  of reinforced basis functions are certainly possible.
The essential point is that they are based on the regression result from the preceding step in the backwards induction and  the preceding iteration.
Our specific choice may be seen as a natural primal choice.
We left flexibility in the choice of the sets $\mathcal{L}^y$, for which depending on the cardinality of the set $\mathcal{L}$, possible choices are the trivial $\mathcal{L}^y = \mathcal{L}$ and $\mathcal{L}^y = \{y\}$, or $\mathcal{L}^y = \mathcal{L}^\prime$ for some set $\mathcal{L}^\prime$ independent of $y$, or more elaborately $\mathcal{L}^y_{j} = \{\varphi(a,y) \;\vert\; a \in K_j(y,x_j)\}$ for some $x_j\in\Rd$.
Note that the use of a step dependent set $\mathcal{L}_{j}^y$ in the above method is straightforward.
\end{remark}

\subsection{Hierarchical reinforced regression algorithm B}\label{sec:modified_reinforced_regression}

Note that~\eqref{eq:simple_least_squares} and~\eqref{eq:adaptive_least_squres} are based on the approximate value functions $v^{(0)}_{j+1}$ and $v^{(i+1)}_{j+1}$, respectively, even though the more accurate approximation $v^{(I)}_{j+1}$ is already available at this point. Hence, we can potentially improve the algorithm's accuracy by always considering the most accurate approximation of the value function $v_{j+1}$ in the Bellman equation.

Fix a number of iterations $I \in \mathbb{N}$ and initialize the approximate value functions at the terminal time by $v^{(i)}_J \equiv v_J$ for all $i \in \{0, ..., I\}$, where $v_J$ is given by (\ref{eq:vJ_initialization}).
The approximate value functions at times previous to $J$ are defined inductively as follows:

Let $j \in \{0, ..., J-1\}$ and assume that $v^{(i)}_{l}$ is already defined for all $l \in \{ j+1, ..., J \}$ and $i\in \{ 0, ..., I\}$.
For $i=0$ and each $y\in\mathcal{L}$ determine the coefficients for the regression basis by solving the least squares problem
\begin{equation}\label{eq:step-0-modified}
\gamma^{(0), y}_{j,1}, ..., \gamma^{(0), y}_{j,K} \;\coloneqq\; \underset{\gamma_1, ..., \gamma_K}{\argmin}{\sum_{m=1}^{M}\left\vert v^{(I)}_{j+1}(y, X^{(m)}_{j+1}) - \sum_{k=1}^{K}\gamma_k\psi_k(X^{(m)}_j)\right\vert^{2}}
\end{equation}
and define the approximate continuation function $c^{(0)}$ by (\ref{eq:simple_continuation_function}).
For $i \in \{1, ..., I\}$ and each $y \in \mathcal{L}$ determine the regression coefficients by solving the least squares problem
\begin{multline}\label{eq:step-1-modified}
\gamma^{(i), y}_{j,1}, ..., \gamma^{(i), y}_{j,K + R^{y}}\coloneqq \underset{\gamma_1, ..., \gamma_{K+R^{y}}}{\argmin}\sum_{m=1}^{M}\bigg\vert\; v^{(I)}_{j+1}(y, X^{(m)}_{j+1}) - \sum_{k=1}^{K}\gamma_k\psi_k(X^{(m)}_j) \\
- \sum_{k=1}^{R^{y}} \gamma_{K+k} v^{(i-1)}_{j+1}(y_k, X^{(m)}_j)\;\bigg\vert^{2},
\end{multline}
where $\{y_k\}_{k=1, ..., R^{y}} = \mathcal{L}^y$, and define the continuation function $c^{(i)}_j$ by (\ref{eq:reinfroced_continuation_function}).

Finally, define the approximation to the value function $v^{(i)}_j$ for all $i =\{0, ..., I\}$ by (\ref{eq:approximate_bellman_principle}).
After ending the backwards induction use $(v^{(I)}_j)_{j=0,...,J}$ as an approximation to the true value function.

\begin{remark}
Note that the identity (\ref{eq:identity_right_lower_triangle}) also holds for the above algorithm.
Moreover, since we are only interested in $v^{(I)}$, we can discard the computation of $c^{(i)}_j$ and $v^{(i)}_j$ for all $0\le j + i \le I-1$, since they do not contribute to the construction of $v^{(I)}$.
The least squares problem then only needs to be solved for $(j,i)\in\{0, ..., J-1\}\times\{0, ..., I\}$ with $0\le j + i \le I-1$ and $0 \le i \le (J-j)\wedge I.$
\end{remark}
\begin{remark}\label{rmk:direct_extension_rr}
Choosing the number of iterations $I = J$ we then have from the previous remark that only the value functions on the diagonal $j = i$ need to be constructed.
In this case, denote $\vmc_j = v^{(j)}_j$, $\cmc_j = c^{(j)}_j$ etc., and observe that the least squares problem which is solved in each step $j \in \{J-1, ..., 0\}$ of the backwards induction is given by
\begin{equation*}
\gamma^{y}_{j,1}, ..., \gamma^{y}_{j, K + R^{y}}\coloneqq \underset{\gamma_1, ..., \gamma_{K+1}}{\argmin}\sum_{m=1}^{M}\bigg\vert\; \vmc_{j+1}(y, X^{(m)}_{j+1}) - \sum_{k=1}^{K}\gamma_k\psi_k(X^{(m)}_j) - \sum_{k=1}^{R^{y}} \gamma_{K+k} \vmc_{j+1}(y_k, X^{(m)}_j)\;\bigg\vert^{2}.
\end{equation*}
where $\{y_k\}_{k=1, ..., R^{y}} = \mathcal{L}^y$.
Hence, for $I=J$ the above algorithm represents a direct extension of the \emph{reinforced regression algorithm} in \cite{belomestny2020optimal} from optimal stopping to optimal control problems.
\end{remark}

\section{Computational cost}
\label{sec:computational-cost}

We study the computational work of the modified reinforced regression
algorithm of Section~\ref{sec:modified_reinforced_regression}. In what follows, the following
operations are considered to be performed at \emph{constant} cost:
\begin{itemize}
\item Multiplications, additions and other primitive operations at cost
  $c_\ast$;
\item Simulation from the distribution of the Markov process $X_j$ at cost
  $c_X$;
\item Evaluation of the standard basis functions $\psi_i$ or of the payoff
  $H_j$ at cost $c_f$;
\end{itemize}

We furthermore introduce the following notations:
\begin{itemize}
\item We set $R \coloneqq \max_{y \in \mathcal{L}} R^y$.
\item The cost of evaluating other non-trivial, but known functions $\varphi$ (think of the value
  function when all the required regression coefficients are already known)
  will be denoted by $\cost(\varphi)$.
\end{itemize}

If an expression involves several such operations, then only the most
expensive constant is reported. (E.g., evaluating a basis function and
multiplying the value by a scalar constant is considered to incur a cost
$c_f$.) We may also use constants $c$ which do not depend on the specifics of
the algorithm. We now go through the individual stages of the algorithm.
\begin{stages}
  \label{stag:cost-stages}
  \begin{enumerate}
  \item Simulating trajectories at cost $\cost_1 = c_X M (J+1)$.
  \item Computing the terminal value function as
    in~\eqref{eq:vJ_initialization} for a given $x\in\Rd$ and all $y\in\mathcal{L}$ at cost
    $\cost_2 = c_f \abs{\mathcal{L}} \abs{\mathcal{K}}$.
  \item For fixed $0 \le j \le J-1$ and $y \in \mathcal{L}$ set up the least
    squares problem~\eqref{eq:step-0-modified} at cost
    $M\left( c_f K + \cost\left( v^{(I)}_{j+1}
      \right)\right)$.\label{item:step-0}
  \item For fixed $0 \le j \le J-1$ and $y \in \mathcal{L}$, we solve the
    least squares problem~\eqref{eq:step-0-modified} at cost $c_\ast M
    K^2$. The total cost is $\cost_4 = c_\ast J M K^2 \abs{\mathcal{L}}$.
  \item For fixed $0 \le j \le J-1$, $y \in \mathcal{L}$, and $1 \le i \le I$
    set up the least squares problem~\eqref{eq:step-1-modified} at cost
    $M\left( c_f K + \cost\left( v^{(I)}_{j+1} \right) + R \cost\left(
        v_{j+1}^{(i-1)} \right)\right)$.\label{item:step-1}
  \item For fixed $0 \le j \le J-1$, $y \in \mathcal{L}$, and $1 \le i \le I$
    solve the least squares problem~\eqref{eq:step-1-modified} at cost
    $c_\ast M (K+R)^2$, leading to a total cost of $\cost_6 = c_\ast M (K+R)^2 J \abs{\mathcal{L}}$.
  \end{enumerate}
  \caption{Stages of the algorithm}
\end{stages}

For simplicity of the presentation, we shall only consider the following scenario:
\begin{assumption}
  \label{ass:L=R}
  The total set of reinforced basis functions contains all available value
  functions, i.e., $\bigcup_{y \in \mathcal{L}} \mathcal{L}^y = \mathcal{L}.$
\end{assumption}

For fixed $0 \le i \le I$ and $0 \le j \le J$ let
\begin{equation}
  \label{eq:vector-values}
  \mathbf{v}^{(i)}_j \coloneqq \left( v^{(i)}_j(y,\cdot) \right)_{y \in
    \mathcal{L}}, \quad \mathbf{c}^{(i)}_j \coloneqq \left( c^{(i)}_j(y,\cdot)
  \right)_{y \in \mathcal{L}}.
\end{equation}
The key step of the cost analysis is understanding the cost of evaluating the
reinforced basis functions, which are, in turn, given in terms of reinforced
basis functions at later time steps. We note that it is essential to analyze
the cost of evaluating the full set of reinforced basis functions
$\mathbf{v}^{(i)}_j$ rather than individual ones $v^{(i)}_j(y, \cdot)$, as the
latter method would show us an apparent explosion of basis functions as we
increase time.\footnote{Suppose that each reinforced basis function
  $v^{(i)}_j(y,\cdot)$ depends on two reinforced basis functions
  $v^{(i-1)}_{j+1}(y', \cdot)$ and $v^{(i-1)}_{j+1}(y'',\cdot)$. If we follow
  this recursion for $l \le i$ steps, we arrive at a total set of $2^l$ basis
  functions. The catch is that many, if not all, of these basis functions
  overlap with basis functions for other reinforced basis functions
  $v^{(i)}_j(\tilde{y}, \cdot)$.}
By~\eqref{eq:approximate_bellman_principle}, evaluating $\mathbf{v}^{(i)}_j$ requires evaluating the payoff functions for all combinations of controls $y \in \mathcal{L}$ and policies $a \in \mathcal{K}$,
then evaluating $\mathbf{c}^{(i)}_j$, and taking the corresponding maxima. In
total, this means
\begin{equation*}
  \cost\left( \mathbf{v}^{(i)}_j \right) \le \abs{\mathcal{K}}
  \abs{\mathcal{L}} (c_f + c_\ast) + \cost\left( \mathbf{c}^{(i)}_j \right).
\end{equation*}
On the other hand, by~\eqref{eq:reinfroced_continuation_function}
evaluating $\mathbf{c}^{(i)}_j$ requires $K$ evaluations of standard basis
functions, $K\abs{\mathcal{L}}$ elementary operations for summing them, one
evaluation of $\mathbf{v}^{(i-1)}_{j+1}$, and $\abs{\mathcal{L}}^2$ elementary
operations for their summation. In total, this means that
\begin{equation*}
  \cost\left( \mathbf{c}^{(i)}_j \right) \le K c_f + K \abs{\mathcal{L}}
  c_\ast + \indic{i>0} \left( \abs{\mathcal{L}}^2 c_\ast + \cost\left(
      \mathbf{v}^{(i-1)}_{j+1} \right) \right).
\end{equation*}
This implies the cost estimate
\begin{equation}
  \label{eq:cost-recursion-vector}
  \cost\left( \mathbf{v}^{(i)}_j \right) \le \abs{\mathcal{K}}
  \abs{\mathcal{L}} (c_f + c_\ast) + K c_f + K \abs{\mathcal{L}}
  c_\ast + \indic{i>0} \left( \abs{\mathcal{L}}^2 c_\ast + \cost\left(
      \mathbf{v}^{(i-1)}_{j+1} \right) \right).
\end{equation}

\begin{lemma}
  \label{lem:cost-reinforced-basis}
  The cost of evaluating $\mathbf{v}^{(i)}_j$, $i=0, \ldots, I$, $j=0, \ldots,
  J$ can be bounded by
  \begin{equation*}
    \cost\left( \mathbf{v}^{(i)}_j \right) \le
    \begin{cases}
      (i+1) \left( \abs{\mathcal{K}} \abs{\mathcal{L}} (c_f+c_\ast) + Kc_f + K
      \abs{\mathcal{L}} c_\ast\right) + i \abs{\mathcal{L}}^2 c_\ast, & j+i
    \le J,\\
    \abs{\mathcal{L}} \abs{\mathcal{K}} (c_f + c_\ast) + (J-j) \left(
      \abs{\mathcal{K}} \abs{\mathcal{L}} (c_f+c_\ast) + Kc_f + (K+1)
      \abs{\mathcal{L}} c_\ast \right), & j+i>J.
    \end{cases}
  \end{equation*}
\end{lemma}
\begin{proof}
  For $a \coloneqq \abs{\mathcal{K}} \abs{\mathcal{L}} (c_f+c_\ast) + Kc_f + K
  \abs{\mathcal{L}} c_\ast$, consider the cost recursion
  \begin{equation*}
    c(k+1) \le a + \abs{\mathcal{L}} c_\ast c(k), \quad k \ge 0.
  \end{equation*}
  Assuming that the recursion hits $i=0$ before $j=J$, i.e., $i+j \le J$, the
  cost $c(k) \coloneqq \cost\left( \mathbf{v}_{j+i-k}^{(k)} \right)$ satisfies
  the recursion with $c(0) \le a$, and, hence, we obtain
  \begin{equation*}
    c(k) \le (k+1) a + k \abs{\mathcal{L}}^2 c_\ast.
  \end{equation*}
  This gives the first expression in the statement of the lemma with $k=i$.

  On the other hand, if $i+j>J$, we hit $j=J$ before $i=0$. In this case,
  $c(k) \coloneqq \cost\left( \mathbf{v}^{(i+j-J+k)}_{J-k} \right)$ satisfies
  the same recursion, but with initial value $c(0) \le \abs{\mathcal{L}}
  \abs{\mathcal{K}} (c_f+c_\ast)$.
\end{proof}

In order to shorten notation, we introduce
\begin{eqnarray*}
  a &\coloneqq& \abs{\mathcal{K}} \abs{\mathcal{L}} (c_f+c_\ast) + Kc_f + K
      \abs{\mathcal{L}} c_\ast,\\
  b &\coloneqq& \abs{\mathcal{L}}^2 c_\ast,\\
  d &\coloneqq& \abs{\mathcal{L}} \abs{\mathcal{K}} (c_f + c_\ast),\\
  e &\coloneqq& \abs{\mathcal{K}} \abs{\mathcal{L}} (c_f+c_\ast) + Kc_f + (K+1)
      \abs{\mathcal{L}} c_\ast,
\end{eqnarray*}
so that the estimate of Lemma~\ref{lem:cost-reinforced-basis} shortens to
\begin{equation*}
  \cost\left( \mathbf{v}^{(i)}_j \right) \le
  \begin{cases}
    (i+1) a + i b, & j+i
    \le J,\\
    d + (J-j) e, & j+i>J.
  \end{cases}
\end{equation*}
We next estimate the cost of setting up the regression
problem~\eqref{eq:step-0-modified}, which is proved similarly.
\begin{lemma}
  \label{lem:cost-step-0-modified}
  The cost of setting up the regression problem for $c^{(0)}_{j}(y, \cdot)$,
  $j=0, \ldots, J-1$, $y \in \mathcal{L}$, can be bounded by
  \begin{equation*}
    \cost_3 \le JMK c_f + M(J-I)\left( (I+1) a + Ib \right) + MId + \half MI(I+1) e.
  \end{equation*}
\end{lemma}

The cost for setting up the least squares problem~\eqref{eq:step-1-modified}
is computed in a similar way.
\begin{lemma}
  \label{lem:cost-step-1-modified}
  The cost of setting up the regression problem for $c^{(i)}_{j}(y, \cdot)$,
  $i=1, \ldots, I$, $j=0, \ldots, J-1$, $y \in \mathcal{L}$, can be bounded by
  \begin{multline*}
    \cost_5 \le JMKc_f + \frac{M}{2} I \left[(I+1)a + (I-1)b\right] (J-I+2) +\\
    + \frac{M}{6} I \left[ 11 a + 2b - 9d +5e + I(I+6)a + 3I(I+b) + 3I^2d +
      I(I+6) e \right].
  \end{multline*}
\end{lemma}
\begin{proof}
  A closer look at~\eqref{eq:step-1-modified} reveals that the total cost of setting
  up all these least squares problems can be bounded by
  \begin{equation}\label{eq:cost-5-sum}
    \cost_5 \le \sum_{j=0}^{J-1} M \left( K c_f + \sum_{i=1}^I \cost\left(
        \mathbf{v}^{(i-1)}_{j+1} \right) \right),
  \end{equation}
  taking into account that $\mathbf{v}^{(I)}_{j+1}$ was already evaluated
  during the set-up of the least squares problem~\eqref{eq:step-0-modified}
  and, hence, does not need to be evaluated again. Using
  Lemma~\ref{lem:cost-reinforced-basis}, we obtain
  \begin{align*}
    \cost_5 &\le JMKc_f + M \sum_{j=0}^{J-1} \left[
              \sum_{i=0}^{(J-j)\wedge(I-1)} \left( (i+1)a +ib \right) +
              \sum_{i=1+(J-j)\wedge(I-1)}^{I-1} \left(d + (J-j-1)e \right)
              \right] \\
    &= JMKc_f + M \sum_{j=0}^{J-I+1} \left[ \sum_{i=0}^{I-1} \left( (i+1)a +
      ib \right) \right] + \\
    &\quad\quad+ M \sum_{j=J-I}^{J-1} \left[ \sum_{i=0}^{J-j} \left((i+1)a +
      ib\right) + \sum_{i=J-j+1}^{I-1} \left(d + (J-j-1)e \right) \right].
  \end{align*}
  Evaluating the double sums gives the estimate from the statement of the lemma.
\end{proof}

Abandoning the difference between $c_f$ and $c_\ast$ using the trivial bounds $\cost_3 \le \const \cost_5$, $\cost_4 \le \const \cost_6$, we obtain
\begin{theorem}
  \label{thr:cost-analysis}
  The computational cost of the  algorithm presented in
  Section~\ref{sec:modified_reinforced_regression} can be bounded by
  \begin{equation*}
    \cost \le \const MJ \left( c_X  + I^2 (K + \abs{\mathcal{K}} +
      \abs{\mathcal{L}}) \abs{\mathcal{L}} + (K+R)^2 \abs{\mathcal{L}}\right),
  \end{equation*}
  where $\const$ is a positive number independent of $\abs{\mathcal{K}}$,
  $\abs{\mathcal{L}}$, $K$, $J$, and $I$.
\end{theorem}

\begin{remark}
  \label{rem:cost-J=I}
  Recall that Remark~\ref{rmk:direct_extension_rr} introduced a significantly
  cheaper variant of  algorithm B for the case
  $I = J$. It is easy to see that the computational cost of this variant is
  bounded by
  \begin{equation*}
    \cost \le \const MJ \left( c_X  + J (K + \abs{\mathcal{K}} +
      \abs{\mathcal{L}}) \abs{\mathcal{L}} + (K+R)^2 \abs{\mathcal{L}}\right),
  \end{equation*}
  i.e., the total cost is proportional to $J^2$ rather than $J^3$. Indeed, the
  main difference in the cost analysis as compared to the full modified
  algorithm is that~\eqref{eq:cost-5-sum} can be replaced by
  \begin{equation*}
    \cost_5 \le \sum_{j=0}^{J-1} M \left( K c_f + \cost\left(
        \mathbf{v}^{(J-j-1)}_{j+1} \right) \right).
  \end{equation*}
  Note that this essentially corresponds to the algorithm
  of~\cite{belomestny2020optimal} directly generalized to optimal control problems.
\end{remark}

\section{Convergence analysis}
\label{sec:convergence-analysis}

In this section we analyze the convergence properties of the standard and
reinforced regression algorithms introduced in the previous sections. For  related convergence analysis in the case of optimal stopping problems we refer  the interested reader to \cite{Z}, \cite{Z1}, and \cite{belomestny2020optimalmk}, see also \cite{bayer2020dynamic}.
Henceforth we assume that
\begin{equation}
\max_{j=0,\ldots,J}\sup_{y\in\mathcal{L}}\sup_{a\in\mathcal{K}}\sup_{x\in\mathcal{X}%
}\left\vert H_{j}(a,y,x)\right\vert \leq C_{H}, \label{eq:boundH}%
\end{equation}
then all the value functions
\[
v_{j}^{\ast}(y,x):=\sup_{\mathbf{A} = (A_\ell)_{\ell = j}^J \in \mathcal{A}_{j}(y,x)}\mathsf{E}\left[
\sum_{\ell=j}^{J}H_{\ell}\left(  A_{\ell},Y_{\ell}(\mathbf{A};j,y),X_{\ell}%
^{j,x}\right)  \right]
\]
are uniformly bounded by $JC_{H}.$
Fix a sequence of
spaces $\Psi_{j},$ $j=0,
\ldots,J,$ of functions defined on $\mathcal{X}.$ We stress that these spaces are not necessarily linear at this point.
Construct the corresponding  sequence of estimates $(v_{j,M}(y,x))_{j=0}^{J}$ via
\begin{align}
v_{J,M}(y,x) &  =\sup_{a\in K_{j}(y,x)}H_{J}(a,y,x)\text{ \ \ }\mathrm{and}%
\label{apdyn}\\
v_{j,M}(y,x) &  =\sup_{a\in K_{j}(y,x)}\left(  H_{j}(a,y,x)+T_W\mathcal{P}%
_{j,M}[v_{j+1,M}](\varphi_{j+1}(a,y),x)\right)  ,\text{ \ \ }j<J,\nonumber
\end{align}
where $\mathcal{P}_{j,M}[g](z,x)$ stands for the  empirical projection of the conditional expectation $\mathsf{E}[  g(z,X_{j+1}^{j,x})]$  on $\Psi_j,$
 based on a  sample
\begin{equation}
\mathcal{D}_{M,j}=\Bigl\{(X_{j}^{(m)},X_{j+1}^{(m)}),\text{ \ \ }m=1,\ldots
,M\Bigr\}\label{sh}%
\end{equation}
from the joint distribution of $(X_{j},X_{j+1}),$ that is,
\[
{\mathcal{P}}_{j,M}[g](z,\cdot)\in\arg\inf_{\psi\in\Psi_{j}%
}\sum_{m=1}^{M}\left[  \left\vert g(z,X_{j+1}^{(m)})-\psi(X_{j}^{(m)}%
)\right\vert ^{2}\right].
\]
In \eqref{apdyn} $T_{W}$ is a truncation operator at level $W=JC_{H}$ defined by
\[
T_{W}f(x)=%
\begin{cases}
f(x), & |f(x)|\leq W,\\
W\mathrm{sign}(f(x)), & \mathrm{otherwise.}%
\end{cases}
\]
Due to Theorem~11.5 in \cite{gyorfi2002distribution}, one has for all $g$ with
$\left\Vert g\right\Vert _{\infty}\leq W,$ $j=0,\ldots ,J-1,$ and all
$z\in\mathcal{L},$ that
\begin{multline}
\mathsf{E}\left[  \left\Vert T_W\mathcal{P}_{j,M}[g](z,\cdot)-\mathsf{E}\left[
g(z,X_{j+1}^{j,\cdot})\right]  \right\Vert _{L_{2}(\mu_{j})}^{2}\right] \label{gyoe} \\
\leq\varepsilon_{j,M}^{2}+2\underset{w\,\in\,\Psi_{j}}{\inf}\left\Vert
g(z,\cdot)-w\right\Vert _{L_{2}(\mu_{j})}^{2}\text{ \ \ with \ \ }%
\varepsilon_{j,M}^{2}:=cW^{4}\frac{1+\log M}{M} \mathtt{VC}(\Psi_j),
\end{multline}
where $\mathtt{VC}(\Psi_j)$ is the Vapnik-Chervonenkis dimension of $\Psi_j$ (see Definition~9.6 in  \cite{gyorfi2002distribution}$), \mu_{j}$ is the distribution of
$X_{j},$  and $c$ is an absolute constant.
In order to keep the analysis tractable, we assume that the sets
$\mathcal{D}_{M,j}$ are independent for different $j$, see Remark~\ref{indep} below.  More specifically, we
consider an algorithmic framework based on (\ref{apdyn}), where for every
exercise date the samples (\ref{sh}) are simulated independently, and consider
the information sets%
\[
\mathcal{G}_{j,M}:=\sigma\left\{  \mathbf{X}^{j;M},\ldots,\mathbf{X}%
^{J;M}\right\}  \text{ with }\mathbf{X}^{j;M}:=\bigl(  X_{j}^{(m)}%
,\, m=\ 1,\ldots,M\bigr)  .
\]
Let us define for $j<J,$ $z\in\mathcal{L},$ $x\in\mathcal{X},$
\begin{equation}
\widehat{C}_{j}(z,x):=T_W\mathcal{P}_{j,M}[v_{j+1,M}](z,x), \label{ctil0}%
\end{equation}
and for a generic (exact) dummy trajectory $\left(  X_{l}\right)
_{l=0,\ldots,J}$ independent of $\mathcal{G}_{j,M},$ let
\begin{equation}
\widetilde{C}_{j}(z,x):=\mathsf{E}_{\mathcal{G}_{j+1,M}}\left[
v_{j+1,M}(z,X^{j,x}_{j+1})\right]  . \label{ctil}%
\end{equation}
Note that $\widetilde{C}_{j}\left(  \cdot,\cdot\right)  $ is a $\mathcal{G}%
_{j+1,M}$-measurable random function while the estimate $\widehat{C}%
_{j}\left(  \cdot,\cdot\right)  $ is a $\mathcal{G}_{j}$-measurable one. We
further define
\begin{equation}
C_{j}^{\ast}(z,x)=\mathsf{E}\left[ v_{j+1}^{\ast}(z,X^{j,x}_{j+1})\right], \quad j<J.
\label{ctil1}
\end{equation}
The following lemma holds.
\begin{lemma}
\label{lemTV} We have that,%
\begin{equation}
\mathsf{E}\left[\Bigl\Vert \sup_{z\in\mathcal{L}}\Bigl\vert \widetilde{C}_{j}(z,\cdot
)-C_{j}^{\ast}(z,\cdot)\Bigr\vert \Bigr\Vert _{L_{2}(\mu_{j})}^2\right]\leq\mathsf{E}\left[\Bigl\Vert \sup_{z\in\mathcal{L}}\Bigl\vert \widehat
{C}_{j+1}(z,\cdot)-C_{j+1}^{\ast}(z,\cdot)\Bigr\vert \Bigr\Vert _{L_{2}%
(\mu_{j+1})}^2\right].\label{boundTV}%
\end{equation}
\end{lemma}

\begin{proof}
Let $X$ be a generic (exact) dummy trajectory independent of $\mathcal{G}%
_{j+1,M}.$ Then from (\ref{ctil}), and (\ref{ctil1}) we see that for $j<J,$%
\begin{equation}
\left\vert \widetilde{C}_{j}(z,X_{j})-C_{j}^{\ast}(z,X_{j})\right\vert
\leq\mathsf{E}_{\mathcal{G}_{j+1,M}}\left[  \left.  \left\vert v_{j+1,M}(z,X_{j+1})-v_{j+1}^{\ast
}(z,X_{j+1})\right\vert \right\vert X_{j}\right]  \label{zw}%
\end{equation}
Next, by (\ref{eq:8}), (\ref{apdyn}), (\ref{ctil0}), and (\ref{ctil1}) we have
that%
\begin{eqnarray}
\nonumber
\left\vert v_{j+1,M}(z,x)-v_{j+1}^{\ast}(z,x)\right\vert &\leq & \sup_{a\in K_{j+1}(z,x)}\left\vert \widehat{C}_{j+1}(\varphi
_{j+2}(a,z),x)-C_{j+1}^{\ast}(\varphi_{j+2}(a,z),x)\right\vert \\
\label{eq:vC}
&  \leq &\sup_{z^{\prime}\in\mathcal{L}}\left\vert \widehat{C}_{j+1}(z^{\prime
},x)-C_{j+1}^{\ast}(z^{\prime},x)\right\vert.
\end{eqnarray}
Hence, by (\ref{zw}) one has that%
\[
\sup_{z\in\mathcal{L}}\left\vert \widetilde{C}_{j}(z,X_{j})-C_{j}^{\ast
}(z,X_{j})\right\vert \leq\mathsf{E}_{\mathcal{G}_{j+1,M}}\left[  \left.  \sup_{z\in\mathcal{L}%
}\left\vert \widehat{C}_{j+1}(z,X_{j+1})-C_{j+1}^{\ast}(z,X_{j+1})\right\vert
\right\vert X_{j}\right] .
\]
Finally, by taking the \textquotedblleft all-in expectation\textquotedblright%
\ w.r.t. the law $\mathcal{\mu}_{j}\otimes\mathbb{P}_{M}$, we observe that%
\begin{eqnarray*}
\mathsf{E}\left[  \sup_{z\in\mathcal{L}}\left\vert \widetilde{C}_{j}%
(z,X_{j})-C_{j}^{\ast}(z,X_{j})\right\vert ^{2}\right] 
&\leq&\mathsf{E}\left\{\mathsf{E}_{\mathcal{G}_{j+1,M}}\left[  \left.  \sup_{z\in\mathcal{L}}\left\vert \widehat
{C}_{j+1}(z,X_{j+1})-C_{j+1}^{\ast}(z,X_{j+1})\right\vert \right\vert
X_{j}\right]\right\}  ^{2}\\
&\leq&\mathsf{E}\left[  \sup_{z\in\mathcal{L}}\left\vert \widehat{C}%
_{j+1}(z,X_{j+1})-C_{j+1}^{\ast}(z,X_{j+1})\right\vert ^{2}\right]
\end{eqnarray*}
by Jensen's inequality and the tower property.
\end{proof}

In fact, Lemma~\ref{lemTV} is the key to the next proposition.

\begin{proposition}
\label{prop_main1}
Set
\[
\mathcal{E}_{j}:=\Bigl\Vert \sup_{z\in\mathcal{L}}\left\vert \widehat{C}%
_{j}(z,\cdot)-C_{j}^{\ast}(z,\cdot)\right\vert \Bigr\Vert _{L_{2}(\mu
_{j}\otimes\mathbb{P}_{M})}, \quad j=0,\ldots,J-1,
\]
with $\mathbb{P}_{M}$ being the law of the sample $X_{j}^{(m)},\, m=\ 1,\ldots,M,\, j=1,\ldots,J.$
Then it holds
\begin{equation}
\label{rec}
\mathcal{E}_{j}
\leq\left\vert \mathcal{L}\right\vert \Bigl(  \varepsilon_{j,M}+\sqrt{2}%
\sup_{z\in\mathcal{L}}\underset{w\,\in\,\,\Psi_{j}}{\inf}\bigl\Vert
\widetilde{C}_{j}(z,\cdot)-w\bigr\Vert _{L_{2}(\mu_{j}\otimes \mathbb{P}_M)}\Bigr)  +\left\vert
\mathcal{L}\right\vert \mathcal{E}_{j+1}
\end{equation}
for all $j=0,\ldots,J-1,$ with $\mathcal{E}_{J}=0$ by definition.
\end{proposition}

\begin{proof}
The case $j=J-1$ follows from \eqref{gyoe} and the fact that $\widetilde C_{J-1}=C^{\ast}_{J-1}.$
Set $r_{j,M}(z)=\underset{w\,\in\,\Psi_{j}}{\inf}\Vert
\widetilde{C}_{j}(z,\cdot)-w\Vert _{L_{2}(\mu_{j}\otimes\mathbb{P}_{M}%
)}.$ Due to (\ref{gyoe}) we have with probability $1,$
\begin{equation}
\mathsf{E}_{\mathcal{G}_{j+1,M}}\left[  \left\Vert \widehat{C}_{j}%
(z,\cdot)-\widetilde{C}_{j}(z,\cdot)\right\Vert _{L_{2}(\mu_{j})}^{2}\right]
\leq\varepsilon_{j,M}^{2}+2r^2_{j,M}(z).
\end{equation}
Hence
\begin{equation}
\Bigl\Vert \widehat{C}_{j}(z,\cdot)-\widetilde{C}%
_{j}(z,\cdot)\Bigr\Vert _{L_{2}(\mu_{j}\otimes\mathbb{P}_{M})}\leq
\varepsilon_{j,M}+\sqrt{2}r_{j,M}(z).\label{tri}%
\end{equation}
 By applying (\ref{tri}) it follows that%
\begin{eqnarray*}
\Bigl\Vert \widehat{C}_{j}(z,\cdot)-C_{j}^{\ast}(z,\cdot)\Bigr\Vert
_{L_{2}(\mu_{j}\otimes\mathbb{P}_{M})}
&\leq &\varepsilon_{j,M}+\sqrt{2}r_{j,M}(z)+\left\Vert \widetilde{C}_{j}(z,\cdot)-C_{j}^{\ast}(z,\cdot)\right\Vert
_{L_{2}(\mu_{j}\otimes\mathbb{P}_{M})}.
\end{eqnarray*}
From this and Lemma~\ref{lemTV} we imply%
\begin{eqnarray*}
\sup_{z\in\mathcal{L}}\Bigl\Vert \widehat{C}_{j}(z,\cdot)-C_{j}^{\ast
}(z,\cdot)\Bigr\Vert _{L_{2}(\mu_{j}\otimes\mathbb{P}_{M})}
&\leq &\varepsilon_{j,M}+\sqrt{2}\sup_{z\in\mathcal{L}}r_{j,M}(z)%
\\
&& +\Bigl\Vert \sup_{z\in\mathcal{L}}\left\vert \widehat{C}%
_{j+1}(z,\cdot)-C_{j+1}^{\ast}(z,\cdot)\right\vert \Bigr\Vert _{L_{2}%
(\mu_{j+1}\otimes\mathbb{P}_{M})}
\end{eqnarray*}
and then (\ref{rec}) follows.
\end{proof}

\begin{corollary}\label{cor:convergence}
Suppose that
\[
\sup_{z\in\mathcal{L}}\underset{w\,\in\,\,\Psi_{j}}{\inf}\bigl\Vert
\widetilde{C}_{j}(z,\cdot)-w\bigr\Vert _{L_{2}(\mu_{j}\otimes \mathbb{P}_M)}\leq\delta, \quad \mathtt{VC}(\Psi_j)\leq D,\quad  0\leq j\leq J-1,
\]
for some $\delta>0$ and $D>0.$
Proposition \ref{prop_main1} then yields for
$j=0,\ldots,J-1,$ by using \eqref{eq:vC}, 
\begin{multline}
\label{eq:main-bound}
\Bigl\Vert \sup_{z\in\mathcal{L}}\left\vert  v_{j,M}(z,\cdot)-v_{j}^{\ast}(z,\cdot)\right\vert \Bigr\Vert _{L_{2}(\mu
_{j}\otimes\mathbb{P}_{M})}\leq
\left( cW^{4}\frac{1+\log M}{M}D+\sqrt{2}\delta\right) \frac{ \left\vert \mathcal{L}%
\right\vert^{J-j+1}-\left\vert \mathcal{L}\right\vert}{\left\vert \mathcal{L}\right\vert
-1}.
\end{multline}
\end{corollary}

\begin{corollary}
By inserting the estimate%
\begin{equation*}
\underset{w\,\in\,\,\Psi_{j}}{\inf}\left\Vert \widetilde{C}_{j}%
(z,\cdot)-w\right\Vert _{L_{2}(\mu_{j}\otimes\mathbb{P}_{M})}
\leq\left\Vert \widetilde{C}_{j}(z,\cdot)-C_{j}^{\ast}(z,\cdot)\right\Vert
_{L_{2}(\mu_{j}\otimes\mathbb{P}_{M})}\!\!\!+\underset{w\,\in\,\,\Psi_{j}%
}{\inf}\left\Vert C_{j}^{\ast}(z,\cdot)-w\right\Vert _{L_{2}(\mu_{j})}%
\end{equation*}
in Proposition \ref{prop_main1}, we get the alternative recursion%
\[
\mathcal{E}_{j}\leq\left\vert \mathcal{L}\right\vert \left(  \varepsilon
_{K,M}+\sqrt{2}\sup_{z\in\mathcal{L}}\underset{w\,\in\,\,\Psi_{j}}{\inf
}\left\Vert C_{j}^{\ast}(z,\cdot)-w\right\Vert _{L_{2}(\mu_{j})}\right)
+\left\vert \mathcal{L}\right\vert (1+\sqrt{2})\mathcal{E}_{j+1},
\]
and under the alternative assumption%
\[
\sup_{z\in\mathcal{L}}\underset{w\,\in\,\,\Psi_{j}}{\inf}\bigl\Vert
{C}^{\ast}_{j}(z,\cdot)-w\bigr\Vert _{L_{2}(\mu_{j})}\leq\delta, \quad \mathtt{VC}(\Psi_j)\leq D,\quad  0\leq j\leq J-1,
\]
for some $\delta>0$ and $D>0,$
we obtain for $j=0,\ldots,J,$ the bounds
\[
\Bigl\Vert\sup_{z\in\mathcal{L}}\left\vert v_{j,M}(z,\cdot)-v_{j}^{\ast
}(z,\cdot)\right\vert \Bigr\Vert_{L_{2}(\mu_{j}\otimes\mathbb{P}_{M})}%
\leq\left(  cW^{4}\frac{1+\log M}{M}D+\sqrt{2}\delta\right)  \left\vert
\mathcal{L}\right\vert \frac{\left(  (1+\sqrt{2})\left\vert \mathcal{L}%
\right\vert \right)  ^{J-j}-1}{(1+\sqrt{2})\left\vert \mathcal{L}\right\vert
-1}.
\]
\end{corollary}

\begin{remark}\label{indep}
In \cite{Z}  the convergence of an  \textquotedblleft
independent sample version\textquotedblright\ of the Longstaff-Schwartz
algorithm was studied based on an assumption similar to the  independence $\mathcal{D}_{M,j}$ for different $j$ here.
However,  in a later
paper \cite{Z1} it was shown (by more involved analysis) that the
convergence rates based on one and the same sample are basically the same as
in \cite{Z} up to certain constants. One therefore may naturally expect that
similar conclusions apply in our context.  Therefore the  numerical examples in Section~\ref{sec:numerical-examples} are based 
on a single sample of $M$ trajectories.  
\end{remark}

The proposed reinforced regression algorithm with $I=J$ uses linear approximation spaces of the form
\begin{eqnarray}
\label{eq:basis-rf}
&\Psi_{j}=\mathrm{span}\{\psi_1(x),\ldots,\psi_K(x), {v}_{j+1,M}(y_{1},x),\ldots,{v}_{j+1,M}(y_{R},x)\},\quad j=0,\ldots,J-1,
\end{eqnarray}
for $\mathcal{L}=\{y_1,\ldots,y_R\},$ where $\psi_1(x),\ldots,\psi_K(x)$ are some fixed basis functions (e.g. polynomials) on $\mathcal{X}.$ In this case $\mathtt{VC}(\Psi_j)\leq K+R,\quad  0\leq j\leq J-1.$ In order to see the advantage of adding additional basis functions more clearly, we prove the following proposition.
\begin{proposition}
\label{prop:cor-main}
Assume additionally that
\begin{equation}
\max_{j=1,\ldots,J}\sup_{y\in\mathcal{L}}\sup_{a\in\mathcal{K}}\sup_{x\in\mathcal{X}}\left|H_{j}(a,y,x_1)-H_{j}(a,y,x_2)\right|\leq L_{H}|x_1-x_2|, \label{eq:lip-boundH}
\end{equation}
for all $x_1,x_2\in \mathcal{X}$ and
\begin{eqnarray}
\label{eq:lip-X}
\max_{j=1,\ldots,J}\max_{\ell=j,\ldots,J}\mathsf{E}[|X_{\ell}^{j,x_1}-X_{\ell}^{j,x_2}|]\leq L_X|x_1-x_2|,\quad \forall x_1,x_2\in \mathcal{X}
\end{eqnarray}
for some constants $L_H>0,$  $L_X>0.$ Then it holds for  reinforced spaces $(\Psi_j)$ from \eqref{eq:basis-rf}
\begin{eqnarray*}
\sup_{z\in\mathcal{L}}\underset{w\,\in\,\,\Psi_{j}}{\inf}\left\Vert
\widetilde{C}_{j}(z,\cdot)-w\right\Vert _{L_{2}(\mu_{j}\otimes \mathbb{P}_M)}\leq JL_XL_H\left[\mathsf{E}\int|X^{j,x}_{j+1}-x|^2\mu_j(dx)\right]^{1/2}.
\end{eqnarray*}
\end{proposition}
\begin{proof}
We have
\begin{eqnarray*}
\sup_{z\in\mathcal{L}}\underset{w\,\in\,\,\Psi_{j}}{\inf}\left\Vert
\widetilde{C}_{j}(z,\cdot)-w\right\Vert _{L_{2}(\mu_{j}\otimes \mathbb{P}_M)}\leq \sup_{z\in\mathcal{L}}\left \|  \mathsf{E}_{\mathcal{G}_{j+1,M}}\bigl[v_{j+1,M}(z,X^{j,\cdot}_{j+1}%
)-v_{j+1,M}(z,\cdot)\bigr]\right \|_{L_2(\mu_j\otimes \mathbb{P}_M)}.
\end{eqnarray*}
 Under assumptions \eqref{eq:lip-boundH} and \eqref{eq:lip-X}, we then have
\begin{eqnarray*}
\max_{j=1,\ldots,J}\sup_{y\in \mathcal{L}}|v^{\ast}_{j}(y,x_1)-v^{\ast}_{j}(y,x_2)|\leq JL_XL_H|x_1-x_2|, \quad \forall x_1,x_2\in \mathcal{X}.
\end{eqnarray*}
By using an additional truncation, one can achieve that the Lipschitz constants of the estimates
 $v_{j,M}(z,\cdot),$ $j=0,\ldots,J-1,$ are all uniformly bounded by a constant $JL_XL_H$ with probability $1.$
\end{proof}
The above proposition implies that if $ JL_H$ stays bounded for $J\to \infty,$ (for example if $H$ scales as $1/J$), then the approximation error  $\underset{w\,\in\,\,\Psi_{j}}{\inf}\Vert
\widetilde{C}_{j}(z,\cdot)-w\Vert _{L_{2}(\mu_{j}\otimes \mathbb{P}_M)}$ becomes small as $J\to \infty.$ Note that the latter property can not be guaranteed when using  fixed (nonadaptive) linear spaces $\Psi_j.$ Of course, the exponential in $J$ factor in \eqref{eq:main-bound} will lead to explosion of the overall error as $J\to \infty,$ but the above observation still indicates that the inclusion of the functions ${v}_{j+1,M}(y_{1},x),\ldots,{v}_{j+1,M}(y_{R},x)$ into $\Psi_j$ can  significantly improve the quality of the estimates $v_{j,M}(z,\cdot)$  especially in the case of large $J$.
Concerning the dependence of the bound \eqref{eq:main-bound} on $J$, we note that this estimate is likely to be too pessimistic, see also a discussion in \cite{Z}. 

\section{Numerical examples}
\label{sec:numerical-examples}

We now present various numerical examples which demonstrate the accuracy of the reinforced regression algorithm in practice. To allow for a direct comparison with the reinforced regression algorithm of~\cite{belomestny2020optimal}, we first consider a (single) optimal stopping problem, more particularly a Bermudan max-call option. Our second example is a multiple stopping problem, for which the hyperparameters already become crucial. Finally, our last example is an optimal control of a gas storage.

We have tested both algorithms A and B for the HRR (hierarchical reinforced regression) method, however, the latter version always gave slightly better results and therefore we have only included the value obtained with the algorithm B.
Intuitively this is to be expected, as the algorithm B uses more accurate regression targets in the backwards induction.
The algorithm A may be of use in situations where one is interested in improving the approximation until a certain accuracy threshold is reached, however, properly done this approach should also include a calculation of upper bounds, which we do not discuss in our paper.

Before, let us also mention how a lower biased estimate to the value of a control problem in a Markovian setting is calculated using the result of a regression procedure.
Let $c$ be an approximation to the function $c^{\ast}$ given by
\begin{equation*}
\copt_j(x,y) = \E \left[v_{j+1}^{\ast}(y, X_{j+1}^{j,x}) \right], \quad x\in\Rd,\; y\in\mathcal{L},\; j \in\{ 0, ..., J-1\},
\end{equation*}
with $\cmc_J \equiv \copt_J \equiv 0$ by convention.
Using the hierarchical reinforced regression method, such an approximation is given by $c^{(I)}$ defined in \eqref{eq:reinfroced_continuation_function}.
Further let $(X^{(m)})_{1 \le m \le M_{\mathrm{test}}}$ be sample trajectories from the underlying Markov chain, generated independently from the samples used in the regression procedure. 
Then we can iteratively define a sequence of polices $(\mathbf{A}^{(m)})_{1 \le m \le M_{\mathrm{test}}}$ with $\mathbf{A}^{(m)} = (A^{(m)}_0, ..., A^{(m)}_J)$ by
\begin{equation*}
A^{(m)}_j := \argmax_{A \in K_j(Y^{(m)}_j, X_j^{(m)})} \left( H_j(A, Y_j^{(m)}, X^{(m)}_j) + c_{j}(\varphi(A, Y^{(m)}_j), X^{(m)}_j)\right),
\end{equation*}
or all $m = 1, ..., M_{\mathrm{test}}$ and $j = 0, ..., J$, where $Y^{(m)}_0 := y_0 \in \mathcal{L}$ and $Y^{(m)}_{j+1} := \varphi_{j+1}(A_{j}^{m},Y^{(m)}_{j})$.
It then follows from the definition, that each $\mathbf{A}^{(m)}$ is an admissible sequence of policies, i.e. $\mathbf{A}^{(m)} \in \mathcal{A}_0(y_0, X^{(m)})$.
Therefore, a lower estimate to the value $\E(v(y_0, X_0))$ is given by
\begin{equation*}
\frac{1}{M_{\mathrm{test}}}\sum_{m=1}^{M_{\mathrm{test}}} \sum_{j = 0}^J H_j(A^{(m)}_j, Y_j^{(m)}, X_j^{(m)}).
\end{equation*}

Lower bounds allow a direct comparison of the performance of different methods, in the sense that the method yielding the highest lower bound (up to Monte Carlo errors) performed best, since this value must be closest to the true value of the control problem.
This direct comparison is, however, not possible for the approximate value $v_0$ which may lie above or below the true value.
Our main premise in the following is that the HRR algorithm is a more efficient way to improve the performance of the regression algorithm as compared with increasing the complexity of the regression basis.
Hence, for this relative comparison it is also sufficient to study lower bounds.

\subsection{Bermudan max-call option}
In this section we evaluate the performance of the hierarchical reinforced regression (HRR) method from Section \ref{sec:altern-algor} on the valuation of a Bermudan max-call option.
Let $(\Omega, \F, (\F_t)_{0\le t \le T}, P)$ be a filtered probability space on which a $d$-dimensional Brownian motion $W = (W(t))_{0 \le t \le T}$ is defined.
Further let $X=(X(t))_{0\le t \le T}$ be the geometric Brownian motion defined by
\begin{equation*}
dX^{k}(t) = (r - \delta) \, X^{k}(t) \, dt + X^k(t)\, \sigma \, dW^k(t),  \quad X^k(0) = x_0, \quad 0 \le t \le T, \quad k \in \{1, ..., d\},
\end{equation*}
where $x_0, r, \delta, \sigma > 0$.
Option rights can be exercised on a predefined set of possible exercise dates $\{t_0, t_1, ..., t_J\}$, where at most one right can be exercised on any given date.
Assume that the exercise dates are equidistant $t_j := j\cdot \Delta t$ for all $j \in \{0, ...,  J\}$ with $\Delta t := T/J$ and define the underlying Markov chain $(X_j)_{j = 0,...,J}$ by $X_j = X({j \Delta t})$.
Recall from Example \ref{ex:multiple-stopping} that in order to model a multiple stopping problem in the optimal control framework we define the set of policies by $\mathcal{K}= \{0,1\}$ and the set of controls by $\mathcal{L}= \{0, ..., y_{max}\}$,
where $y_{max}$ is the number of exercise rights.
Further, we define
\begin{equation*}
  \varphi_j(a,y) \coloneqq (y-a)_+, \quad K_j(x, y) \coloneqq \{0, 1\wedge y\}, \quad H_j(a, y, x) \coloneqq a \cdot g_j(x) \, e^{-t_j r},
\end{equation*}
for all $y\in\mathcal{L}$, $a \in \mathcal{K}$ and $j \in \{0, ..., J\}$, where $g$ is the max-call pay-off function defined by
\begin{equation*}
g(x) \coloneqq (\max\{x^1, ...., x^d\} - C)_+, \quad x\in\Rd,
\end{equation*}
where $C \in \mathbb{R}_+$ is the option strike.
Then the value function $v^*_0(y_{max}, x_0)$ defined in \eqref{eq:7}  yields the value of the Bermudan max-call option with underlying $X$ and data $(d, J,T, y_{max}, x_0, C, r, \delta, \sigma)$.

 \subsubsection{Single exercise right}
We will first consider the case of a single exercise right $y_{max} = 1$.
This is a standard example in the literature, see for example \cite{broadie1997stochastic, J_Rogers2002, J_AB2004} and more recently \cite{becker2019deep}.
The performance of the reinforced regression method for this example was already analyzed in \cite{belomestny2020optimal}. 
We revisit this example in order to demonstrate that even in the optimal stopping case our novel HRR method allows for improvements in computational costs without sacrificing the quality of the estimations.

\begin{table}[!htbp]\centering
\begin{tabular}
[c]{c|c|c|c|c|c}
\multirow{3}{*}{$d$}		& \multirow{3}{*}{Basis} & \multicolumn{3}{c|}{Lower bounds} & \multirow{3}{*}{CI from \cite{becker2019deep}} \\ \cline{3-5}
								&						& Regression			& HRR						& Reinf. Reg.			& \\\cline{3-5}
								&						& $I=0$					& $I=1$					& $I=9$					& \\ \hline
\multirow{4}{*}{2}		& $\Psi_1$ 		& 13.015 (0.022)	& 13.772 (0.015)	& 13.794 (0.015)	& \multirow{4}{*}{[13.880,13.910]}\\
								& $\Psi_{1,g}$	& 13.679 (0.019)	& -							&  -							& \\
								& $\Psi_2$		& 13.775 (0.016)	& 13.871 (0.015)	&  13.882 (0.014)	& \\
								& $\Psi_3$		& 13.874 (0.016)	& -							& -							& \\\hline
\multirow{4}{*}{3}		& $\Psi_1$		& 17.764 (0.029)	& 18.526 (0.017)	& 18.540 (0.018)	& \multirow{4}{*}{[18.673,18.699]}\\
								& $\Psi_{1,g}$	& 18.404 (0.022)	& -							& -							& \\
								& $\Psi_2$		& 18.519 (0.020)	& 18.639 (0.017)	& 18.653 (0.017)	& \\
								& $\Psi_3$		& 18.655 (0.021)	& -							& -							& \\\hline
\multirow{3}{*}{5}		& $\Psi_1$		& 25.463 (0.024)	& 25.998 (0.021)	& 25.990 (0.019)	& \multirow{3}{*}{[26.138, 26.174]} \\
								& $\Psi_{1,g}$	& 25.823 (0.026)	& -							& -							& \\
								& $\Psi_2$		& 25.990 (0.023)	& 26.097 (0.019)	& 26.109 (0.020)	& \\
								& $\Psi_3$		& 26.111 (0.022)	& -							& -							& \\\hline
\multirow{3}{*}{10}		& $\Psi_1$		& 38.022 (0.025)	& 38.234 (0.022)	& 38.225 (0.024)	& \multirow{3}{*}{[38.300,38.367]} \\
								& $\Psi_{1,g}$	& 38.058 (0.024)	& -							&  -							& \\
								& $\Psi_2$		& 38.299 (0.023)	& 38.316 (0.020)	&  38.331 (0.021)	& \\
								& $\Psi_3$		& 38.349 (0.021)	& -							& -							& \\\hline
\end{tabular}
\vspace{.2cm}
\caption{Lower bounds ($\pm$ 99.7\%  Monte-Carlo error) for the value of the Bermudan max-call option with data $J=9$, $T=1$, $y_{max}=1$, $x_0 = C = 100$, $r=0.05$, $\delta=0.1$, $\sigma=0.2$ and different numbers of underlying assets $d \in \{2, 3, 5, 10\}$.
For all methods we used $M=10^6$ training sample paths and $M_{\mathrm{test}}=10^7$ paths for calculating the lower bound.
The last column presents the 95\% confidence intervals for the value of the Bermuda option from \cite{becker2019deep}.}
\label{tab:max_call_bounds}
\end{table}

Define the functions $f_i: \Rd \to \R$, $x\mapsto \operatorname{sort}(x^{1}, \ldots, x^{d})^{i}$, the $i$th largest entry of $x$, for $i \in \{1, \dots, d\}$
and consider the following three sets of regression basis functions:
\begin{equation*}
\begin{array}{c}
\Psi_1 \coloneqq \{ 1, f_1, ..., f_d \}, \quad \Psi_{1,g} \coloneqq \Psi_1 \cup \{g\}, \quad  \Psi_2 \coloneqq \Psi_1 \cup\{ f_i \cdot f_j \;\vert\; 1 \le i \le j \le d\} \\
\Big.\Psi_3 \coloneqq \Psi_1 \cup \Psi_2 \cup\{ f_i \cdot f_j \cdot f_k \;\vert\; 1 \le i \le j \le k \le d\}.
\end{array}
\end{equation*}
Note that the cardinalities of these sets are given by $|\Psi_1|=d+1$, $|\Psi_{1,p}|=d+2$, $|\Psi_2|=\frac{1}{2}d^2+\frac{3}{2}d +1$, and $|\Psi_3|=\frac{1}{6}d^3 + d^2 + \frac{11}{6}d+1$, respectively.
Regarding the HRR method, we use the algorithm of the second type described in Section \ref{sec:modified_reinforced_regression}.
Further note that in the optimal stopping case there is only one choice for the set of reinforced value function for the HRR method since $\mathcal{L}=\{1\}$ and therefore we always set $\mathcal{L}^1 = \{1\}$.

\begin{figure}[!htbp]\centering
\includegraphics[width=0.9\textwidth]{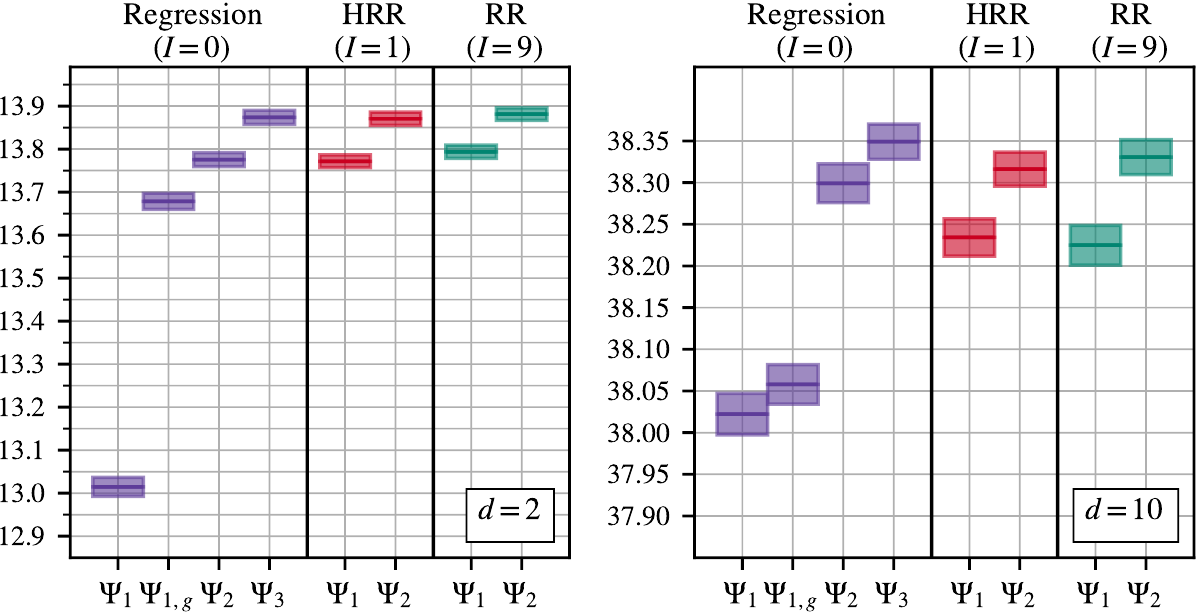}
\caption{A visualization of the lower bounds from Table \ref{tab:max_call_bounds}.}\label{fig:max_call_lower_bounds}
\end{figure}

We considered two different set-ups for the comparison of the different methods: 
\begin{itemize}
\item First we keep the number of exercises dates $J$ fixed and vary the number of underlying assets $d$;
\item Second we keep $d$  fixed and vary $J$ (while also keeping $T$ fixed).
\end{itemize}
 
In Table \ref{tab:max_call_bounds} we present lower estimates to the value of a Bermudan max-call option with a single exercise right for $J=9$ and $d \in \{2, 3, 5, 10\}$.
In the corresponding Figure \ref{fig:max_call_lower_bounds} we have visualized the lower bounds for the comparison between the different regression methods.
For each of the considered methods we used $M=10^6$ simulated training samples paths to determine the regression coefficients and $M_{\text{test}} = 10^7$ paths for calculating the lower bounds.

In order to give the reader an easy reference point, in Table \ref{tab:max_call_bounds} we have also included the confidence intervals for the value of the Bermudan max-call option from \cite{becker2019deep}.
Note however that we do not aim for an improvement of the latter values in terms of benchmarking.
In fact the method used in \cite{becker2019deep} is quite different from ours, as it uses deep neural networks for approximating the optimal stopping policies at each time step, and a direct comparison would require the usage of higher order polynomials for our method.

We first observe that across all numbers of assets $d$ the HRR method with the set of basis functions $\Psi_1$ performs significantly better then the standard regression method with the basis $\Psi_1$ and $\Psi_{1,g}$.
The same holds true when comparing the methods using the regression basis $\Psi_2$.
More importantly however, we observe that for $d\le 5$ the HRR method with basis $\Psi_1$ yields lower bounds of the same quality as obtained with the standard method and the larger basis $\Psi_2$.
The same holds true when comparing the HRR method with basis $\Psi_2$ against the standard method with the basis $\Psi_3$.
In the case $d=10$ assets, the lower bounds obtained with the HRR method and basis $\Psi_1$ lie just slightly below the values of the lower bounds obtained with standard method and the basis $\Psi_2$, however one has to keep in mind that in this case $|\Psi_1|=11$ and $|\Psi_2|=286 $.
Moreover, we see that the HRR method with a recursion depth $I=1$ performs just as well as the (full depth) reinforced regression method ($I = J=9$).

\begin{figure}[!htbp]\centering
\includegraphics[width=0.7\textwidth]{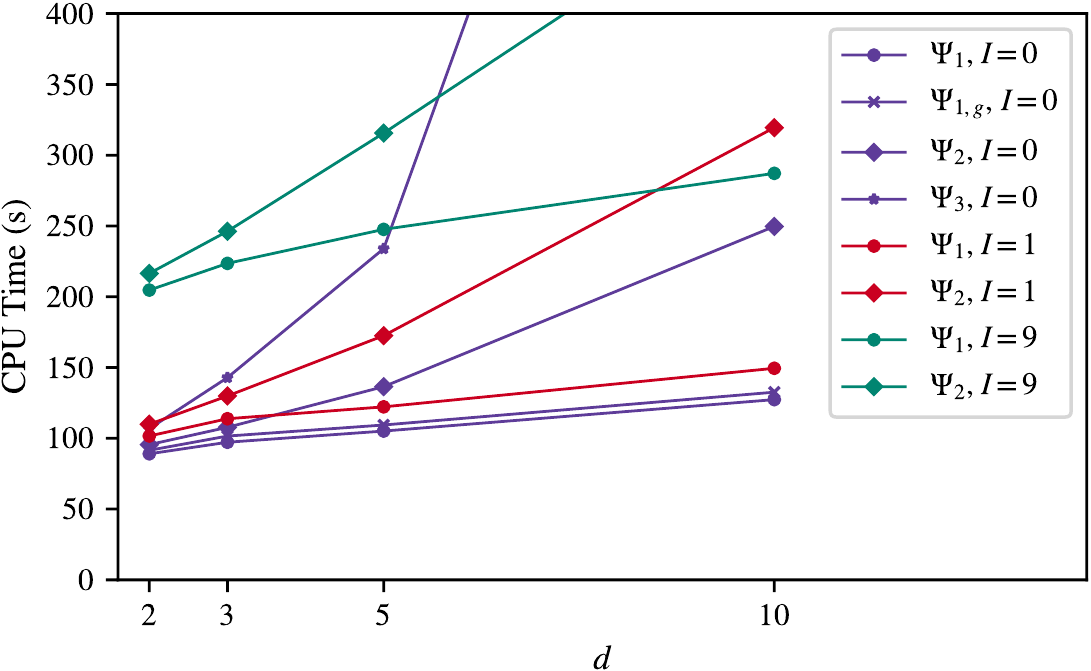}
\caption{The elapsed CPU times during the backwards induction and calculation of the lower bounds from Table \ref{tab:max_call_bounds}, plotted with respect to the number of underlying assets $d$.}
\label{fig:max_call_cpu}
\end{figure}

Furthermore, in Figure \ref{fig:max_call_cpu} we have visualized the corresponding elapsed CPU times during the backwards induction and the calculation of the lower bounds.
As foreshadowed in Section \ref{sec:computational-cost}, we see that the computational costs of the HRR method are significantly reduced by choosing a small recursion depth $I$. 
In particular, we are able to state that for sufficiently large $d$ ($d \ge 5$ respectively $d\ge 3$) the HRR method with recursions depth $I=1$ and the basis $\Psi_1$ respectively $\Psi_2$ is more efficient then the standard method with the basis $\Psi_2$ respectively $\Psi_3$.

\begin{figure}[htp!]\centering
\includegraphics[width=0.7\textwidth]{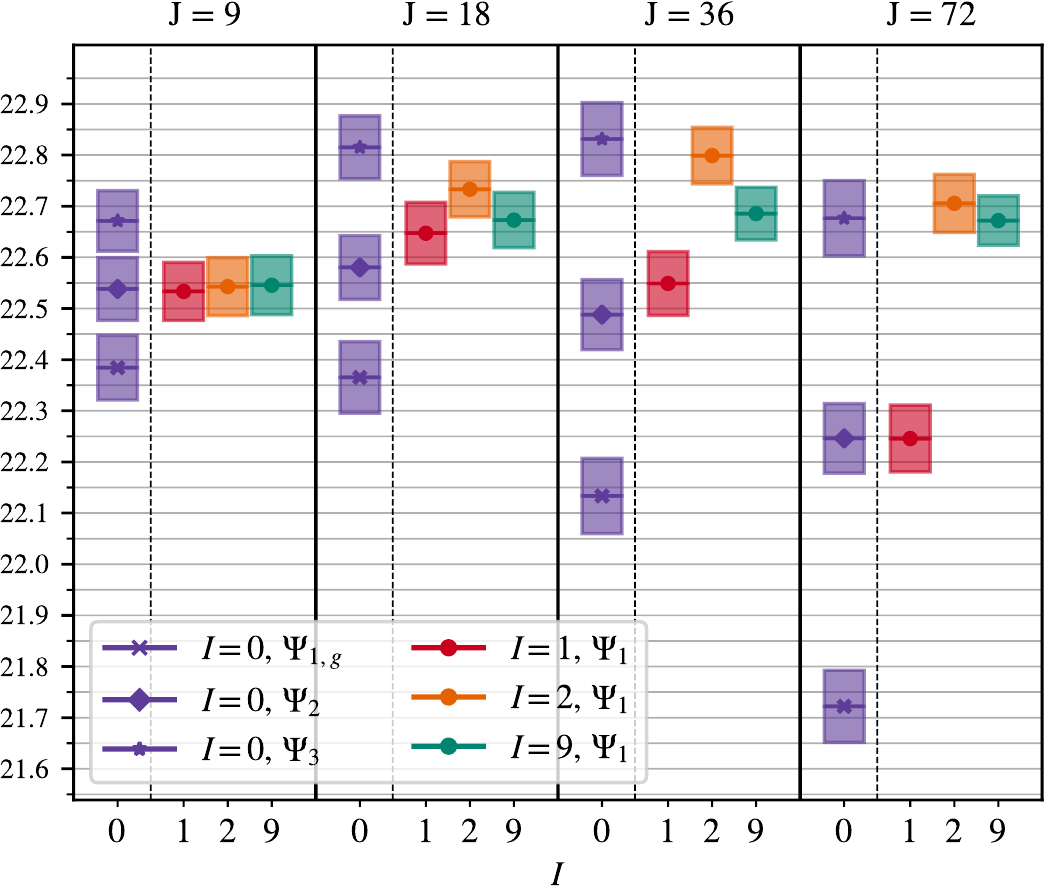}
\caption{Visualization of the lower bounds ($\pm$ 99.7\% Monte-Carlo error) of the values of Bermuda max-call options with $J=9, 18, 36, 72$ exercise dates and $d=4$, $T=1$, $y_{max}=1$, $x_0=C=100$, $r=0.05$, $\delta=0.1$, $\sigma=0.2$.
The values are obtained with the standard regression method $I=0$ and the HRR method $I=1, 2, 9$.
For all methods we used $M=10^6$ training sample paths and $M_{\mathrm{test}}=10^6$ sample paths for the calculation for the lower bounds.}
\label{fig:max_call_increasing_j}
\end{figure}

The results of the second set-up are visualized in Figure \ref{fig:max_call_increasing_j}.
In this case, we have approximated the value of Bermudan max-call options with a fixed number of assets $d=4$ and different numbers of exercise dates $J \in \{9, 18, 36, 72\}$, while also keeping $T=1$ fixed.
When keeping all other parameters fixed, the value of the option clearly is non-decreasing in the number of exercise dates $J$.
Our first observation is that for all considered methods there exists a threshold for $J$ at which the performance worsens. Indeed, Corollary~\ref{cor:convergence} shows that the approximation error depends exponentially on $J$.

However, in practice, some methods are less vulnerable to the error explosion in $J$ than others.
In this example, we see that the standard regression method with the basis $\Psi_{1,g}$ and $\Psi_{2}$, respectively,  starts to perform worse for $J\ge 36$.
The lower bounds calculated with the HRR method with the basis $\Psi_{1}$ and $I=1$ stay approximately on the same level as the lower bounds calculated with the standard method and the basis $\Psi_{2}$, for all numbers of exercise dates.
The lower bounds calculated with the standard method and the basis $\Psi_3$ first increase when moving from $9$ to $18$ exercise dates and decrease at last when moving from $36$ to $72$ exercise dates.

The main observation is that when we increase $J$, the HRR methods with $I=2$ and $I=9$ come closer to the lower bounds calculated with the standard method and the basis $\Psi_3$.
This underlines the theoretical discussion for $J\to\infty$ from Section \ref{sec:convergence-analysis}.
Moreover, we see that the HRR method performs at least as well with $I=2$ as with $I=9$.
Regarding CPU time, the HRR method with $I=2$ is more efficient than then the standard regression method with the basis $\Psi_3$ for $J\ge9$.
Comparing the HRR methods with $I=2$ and $I=9$, we see that choosing the parameter $I$ small is necessary to in order to obtain desirable efficiency.
Since on the other hand, the HRR method with $I=1$ performed significantly worse then with $I=2$, we also see that in this case it was necessary to choose $I>1$.
These observations underline the relevance of the HRR in its full complexity even in the case of optimal stopping problems.

\subsubsection{Multiple exercise rights}
Next we consider a Bermudan max-call option with $y_{max} = 4$ exercise rights.
In this case the HRR method allows for different possibilities of reinforced value functions depending on the choice of the sets $\mathcal{L}^y$ (recall Remark \ref{rem:choice_of_Ly}).
Since $y_{max}$ is small, we choose $\mathcal{L}^y \equiv \{1, 2, 3, 4\}$ for simplicity.

\begin{table}[h]\centering
\begin{tabular}
[c]{c|c|c|c|c|c}
\multirow{2}{*}{Basis} & Regression & \multicolumn{4}{c}{Hierarchical Reinforced Regression} \\\cline{2-6}
& $I=0$ & $I=1$ & $I=2$ & $I=3$ &   $I=5$ \\ \hline 
$\Psi_1$ & 90.863 (0.072) &  92.038 (0.070) & 92.287 (0.070)& 92.311 (0.067) & 92.357 (0.061)\\ 
$\Psi_{1,g}$ & 91.837 (0.082) & - & - & - & - \\
$\Psi_2$ & 92.140 (0.070) & 92.418 (0.064) & 92.548 (0.060)& 92.631 (0.061) & 92.625 (0.061)\\
$\Psi_3$ & 92.571 (0.069) & - & - & - & -  \\
\end{tabular}
\vspace{.2cm}
\caption{Lower bounds ($\pm$ 99.7\%  Monte-Carlo error) for the value of the Bermudan max-call option with data $J=24$, $T=2$, $y_{max}=4$, $x_0 = C = 100$, $d=5$, $r=0.05$, $\delta=0.1$, $\sigma=0.2$.
For all methods we used $M=10^6$ training sample paths and $M_{\mathrm{test}}=10^7$ paths for calculating the lower bound.
An upper bound to the value, calculated with the dual approach from \cite{J_Schoen2010} and  \cite{J_BenSchZha}, is given by $92.971$ $(0.043)$, with the HRR method with $I=3$ and basis $\Psi_2$ using $10^5$ outer and $10^3$ inner sample paths.}
\label{tab:multiple_max_call_bounds}
\end{table}

In Table \ref{tab:gas_storage_results}
we present lower bounds to the value of a Bermuda max-call option with $y_{max}=4$ exercises rights, obtained with the standard regression method and HRR method for different choices of regression basis functions and the parameter $I$, with the implementation of the second type described in Section \ref{sec:modified_reinforced_regression}.
We first observe that for a fixed set of basis functions $\Psi_1$ or $\Psi_2$ increasing the parameter $I$ yields increased, and thus improved, lower bounds.
This improvement is most significant when moving from $I=0$ (standard regression) to $I=1$ and from $I=1$ to $I=2$ and becomes less significant when further increasing $I$.
Moreover, we observe that the HRR method with $I=1$ and basis functions $\Psi_1$ yields better lower bounds then the standard regression method with the larger set of basis functions $\Psi_{1,g}$ and more importantly, for $I\ge 2$ the HRR with basis functions $\Psi_1$ method yields better lower bounds than the standard regression method with the even larger set of basis functions $\Psi_{2}$.
This observation prevails when comparing the standard regression method with the basis $\Psi_3$ against the HRR method with the basis $\Psi_2$.
We can therefore conclude that the HRR method yields results of better quality than standard regression using fewer regression basis functions.
Moreover, we realize that up to changes that are insignificant with respect to the Monte Carlo error, the HRR reaches its best performance already for $I=3$, thus further increasing $I$ is not necessary.

\newcommand{\dd}{\mathrm{d}}
\subsection{A gas storage problem}
In this subsection we consider a gas-storage problem of the kind introduced in Example \ref{ex:gas-storage}.
In contrast to the example in the previous subsection, this optimal control problem is not of a multiple stopping type, which is a consequence of the anti-symmetry in the policy set: injection of gas into the facility $(a=1)$, no action $(a=0)$ and production of gas $(a=-1)$.

For the gas price we use a similar but slightly more elaborate model to the one proposed in \cite{thompson2009natural} (and also used in \cite{gyurko2011monte}).
More specifically, we use the following joint dynamics to model the \emph{price of crude oil} $X^1$ and the \emph{price of natural gas} $X^2$
\begin{equation}\label{eq:oil_gas_price_dynamics}
\begin{split}
\dd X^1(t) &= \alpha_1(\beta - X^1(t)) \dd t + \sigma_1 X^1(t) \dd W^1(t) + \left(J^1_{N(t-)+1} - X^1(t) \right)\dd N(t) \\
\dd X^2(t) &= \alpha_2(X^1(t) - X^2(t)) \dd t + \sigma_2 X^2(t) \dd W^2(t) + \left( J^2_{N(t-)+1} - X^2(t) \right)\dd N(t),
\end{split}
\end{equation}
for $0 \le t \le T$, where $\beta, \alpha_i, \sigma_i > 0$ for $i=1,2$, $W^1$ and $W^2$ are  Brownian motions with correlation $\rho_W \in [0,1]$, $N$ is a Poisson process with intensity $\lambda>0$ and $(J_k)_{k = 1, ...}$ are i.i.d. normal distributed random vectors with $J^i_1 \sim \mathcal{N}(\mu_i, \eta_i^2)$, $\mu_i, \eta_i >0$ and $\rho_J = \mathrm{Cor}(J_1^1, J_1^2) \in [0,1]$.
Moreover we assume that $(W^1, W^2)$, $N$ and $(J^1,J^2)$ are independent.
Note that both $X^1$ and $X^2$ are mean reverting processes with jump contributions.
The oil price process $X^1$ reverts to the long-term constant mean $\beta$ and the gas price process $X^2$ reverts towards the oil price $X^1$, which is aiming to model the well known strong correlation between crude oil and natural gas prices. 
Note also that we have assumed for simplicity that the jump signal, which has the purpose of modeling price peaks, is the same Poisson process for both oil and gas prices, however the magnitude of the jumps is given by different (but correlated) normal distributed random variables.

Denote by $(\widetilde{X}_j)_{j = 1, ..., 365}$ the 2-dimensional Markov chain that is obtained by discretizing the above SDE \eqref{eq:oil_gas_price_dynamics} with an Euler-scheme on the time interval interval $[0,1]$.
We assume that the manager of the gas storage facility has the possibility to buy and sell gas on a predefined set of dates in the year $\{t_j\}_{j = 1, ..., J} \subset \{1, ..., 365\}$ with $t_j = j \cdot \delta t$ and some $\delta t, J \in \mathbb{N}$ such that $\delta t \cdot J \le 365$.
The 2-dimensional Markov chain underlying the optimal control problem is then given by $X = (X_j)_{j = 0,..., J}$ with $X_j := \widetilde{X}_{t_j}$.

Recall from Example \ref{ex:gas-storage} that we assume that the volume of gas in the storage can only be increased or decreased by a fraction $\Delta = 1/N$ for some $N\in\mathbb{N}$ over the time interval of $\delta t$ days.
The state space of the control variable is then given by $\mathcal{L} = \{0, \Delta, 2\Delta, ..., 1\}$.
Also recall the definition of the space of policies $\mathcal{K}$, the constraint sets $K_j$ and the function $\varphi_j$ from Example \ref{ex:gas-storage}.
We assume that there is no trading at $j=0$ hence $K_0 \equiv \{0\}$.
The cash-flow underlying to the optimal control problem only depends on the second component of the Markov chain $X_j$ and is given by
\begin{equation*}
H_j(a, y, x) = - a\cdot \Delta \cdot x^2 \cdot e^{- r  j(\delta t /365)}, \quad a\in\mathcal{K}, \, y\in\mathcal{L},\; j = 1, ..., J,
\end{equation*}
where $r>0$ is the interest rate.

We do not pay attention to the physical units of the parameters quantifying the gas storage capacity and the quotation of the gas price, since the linearity of the pay-off with respect to the parameter $\Delta$ and the gas price $X^2_j$ allows to properly scale the resulting value of the optimal control problem.
The following specific choice of the price model parameters are oriented at the values in \cite{thompson2009natural}
\begin{equation}\label{eq:parameter_choice_gas}
\begin{array}{c}
\beta = 45, \quad \alpha_1 = 0.25, \quad  \alpha_2 = 0.5,  \quad \sigma_1= \sigma_2 = 0.2,\quad \rho_W = 0.6,\\
 \quad  \lambda = 2, \quad \mu_1 = \mu_2 = 100, \quad \eta_1 = \eta_2 = 30, \quad \rho_J = 0.6.
\end{array}
\end{equation}
Figure \ref{fig:oil_gas_price_sample} shows a sample trajectory of the Markov chain $X$ with the above parameters.

\begin{figure}[!htbp]\centering
\includegraphics[width=0.5\textwidth]{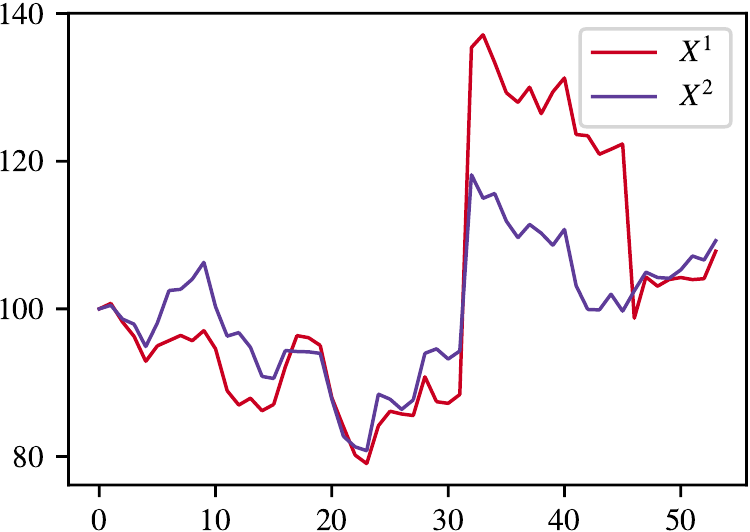}
\caption{A sample path of the Markov chain $(X_j)_{j=0,...,J} = (\widetilde{X}_{t_j})_{j=0,...,J}$
where $t_j = j\cdot 7$ and $J=52$. 
The approximation $\widetilde{X}=(\widetilde{X}_j)_{j=1, ..., 365}$ to the SDE \eqref{eq:oil_gas_price_dynamics} is simulated with the parameters given in \eqref{eq:parameter_choice_gas} and $\widetilde{X}_0 = (100, 100)$. 
$X^1$ and $X^2$ serve as models for the prices of crude oil and natural gas.}\label{fig:oil_gas_price_sample}
\end{figure}

Further we define the following sets of polynomial regression basis functions
\begin{equation}\label{eq:gas_basis functions}
\begin{split}
P_i(X^2) &:= \left\{ (x_1, x_2) \mapsto (x_2)^p \;\big\vert\; p = 0, ..., i \right\} \\
P_i(X^1, X^2) &:= \left\{ (x_1, x_2) \mapsto (x_1)^p (x_2)^q \;\big\vert\; p, q = 0, ..., i, \;\; p + q \le i  \right\}.
\end{split}
\end{equation}

\begin{table}[!htbp]\centering
\begin{tabular}{c | c | c | c}
I & Basis & $v_0(Y_0, X_0)$ & Lower bounds \\\hline
\multirow{5}{*}{0} & $P_1(X^2)$ & 78.381 & 70.489 (0.066) \\
& $P_1(X^1, X^2)$ & 78.575 & 70.635 (0.068) \\
& $P_2(X^2)$ & 73.072 & 71.253 (0.068) \\
& $P_2(X^1, X^2)$ & 73.207 & 71.402 (0.068) \\
& $P_3(X^1, X^2)$ & 72.929 & 71.333 (0.081) \\
& $P_4(X^1, X^2)$ & 72.595 & 71.498 (0.068) \\\hline
1 & $P_1(X^1, X^2)$ & 71.991 & 71.579 (0.070) \\
\end{tabular}
\vspace{.2cm}
\caption{Approximate values and lower bounds for the gas storage problem with parameters given in \eqref{eq:gas_example_params} and price model parameters given in \eqref{eq:parameter_choice_gas}.
Note that - although seemingly so - the estimate $v_0$ not necessarily presents an upper bound to the true value and is included in the table only for verification purposes.}
The quantities were obtained with the standard regression method ($I=0$) and the HRR method ($I=1$), the different sets of basis functions  \eqref{eq:gas_basis functions}, $M=10^5$ training sample paths and $M_{\mathrm{test}}=10^6$ paths for calculating the lower bounds.\label{tab:gas_storage_results}
\end{table}

We have approximated the value of the gas storage problem with the following parameters
\begin{equation}\label{eq:gas_example_params}
\delta t = 7,\quad J = 52, \quad \Delta = 1/8, \quad X_0 = (100, 100), \quad Y_0 = 4/8, \quad r = 0.1.
\end{equation}
In this configuration the gas storage facility is initially loaded with half its capacity and the gas storage manager has the possibility to trade gas every seven days, and the amount by which the manager can inject or produce gas is one height of the total capacity.
In Table \ref{tab:gas_storage_results}
we present the numerical results that were obtained with the standard regression method and the HRR method.

We used $M=10^5$ training sample paths and $M_{test}=10^6$ sample paths for the calculation of the lower bounds.
For the set of reinforced basis functions in the HRR method we have chosen $\mathcal{L}^y \equiv \{Y_0\}$, i.e. in each step of the backwards induction, the regression basis was reinforced with only one function. 

We observe at first, that the lower bounds obtained with the standard regression method are improved when using polynomials in both variables $(X^1,X^2)$ instead of just in the second variable $X^2$ (gas price) and are also improved when using polynomials of increasing order (with the only exception of the third degree polynomials).
Moreover, we observe that the lower bound obtained with the HRR method, using the set of basis functions $P_1(X^1, X^2)$ and $I=1$, lies above all lower bounds that were obtained with the standard regression, in particular the bound obtained with the regression basis $P_4(X^1, X^2)$ (up to Monte Carlo errors).
Hence the HRR method based on polynomials of degree one performed at least as well as the standard method with polynomials of degree four.

\subsection{Conclusions}
\label{sec:conclusions-1}

Let us summarize the findings of the numerical experiments.
We observe that the hierarchical reinforced regression algorithm (HRR) based on polynomial basis functions of a certain degree $\deg$ tend to produce results comparable to standard regression (SR) based on polynomial basis functions of degree $\deg + 1$ or even higher, see Figures~\ref{fig:max_call_lower_bounds},~\ref{fig:max_call_increasing_j}, Tables~\ref{tab:multiple_max_call_bounds}, and, most impressively, \ref{tab:gas_storage_results}.

The numerical results also indicate that, indeed, HRR with low depth of the hierarchy $I$ already performs very well, even if $I \ll J$, see Figures~\ref{fig:max_call_lower_bounds},~\ref{fig:max_call_increasing_j} and Table~\ref{tab:multiple_max_call_bounds}. Hence, HRR performs with similar accuracy to the reinforced regression algorithm (RR) of \cite{belomestny2020optimal}, but at much improved cost. Additionally, when comparing HRR with SR at fixed accuracy, the computational cost of HRR is usually much smaller, especially for $d$ large, see Figure~\ref{fig:max_call_cpu}.

Finally, we note that the accuracy of the HRR method increases substantially when the time discretization is refined, i.e., when $J$ is increased for fixed time horizon $T$. This theoretically very plausible observation (see Section~\ref{sec:convergence-analysis}) is backed up by numerical experiments, see Figure~\ref{fig:max_call_increasing_j}.

\bibliographystyle{alpha}
\bibliography{references}
\end{document}